\documentclass[11pt,twoside]{article}
\usepackage{times}
\usepackage{amsmath,amssymb}
\usepackage{color}
\usepackage{amsthm}
\usepackage{tikz}
\usepackage{hyperref}
\usepackage[toc,page,title,titletoc,header]{appendix}

\allowdisplaybreaks
 \def \no{\nonumber}

\def\R {\Bbb R}

\def \Z {\Bbb Z}
\def\p{\partial}
\def\ve{\varepsilon}
\def\f{\frac}

\def\al{\alpha}
\def\t{\tilde}

\def\g{\gamma}

\def\dl{\delta}

\def\ds{\displaystyle}

\newcommand{\md}{\mathrm{d}}
\newcommand{\lc}{\left(}
\newcommand{\rc}{\right)}
\newcommand{\lp}{\left\|}
\newcommand{\rp}{\right\|}
\newcommand{\crit}{\text{crit}}
\newcommand{\conf}{\text{conf}}
\newcommand{\ra}{[T_0,\infty)\times\mathbb{R}^n}
\newcommand{\ral}{[T_0,t)\times\mathbb{R}^n}
\newcommand{\rb}{[1,2]\times\mathbb{R}^n}
 \allowdisplaybreaks

\begin{document}
 \footskip=0pt
 \footnotesep=2pt
\let\oldsection\section
\renewcommand\section{\setcounter{equation}{0}\oldsection}
\renewcommand\thesection{\arabic{section}}
\renewcommand\theequation{\thesection.\arabic{equation}}
\newtheorem{claim}{\noindent Claim}[section]
\newtheorem{theorem}{\noindent Theorem}[section]
\newtheorem{lemma}{\noindent Lemma}[section]
\newtheorem{proposition}{\noindent Proposition}[section]
\newtheorem{definition}{\noindent Definition}[section]
\newtheorem{remark}{\noindent Remark}[section]
\newtheorem{corollary}{\noindent Corollary}[section]
\newtheorem{example}{\noindent Example}[section]

\title{Global existence for small amplitude semilinear wave equations
 with  time-dependent scale-invariant damping}
\author{Daoyin He$^{1*}$, \qquad
Yaqing Sun$^{2*}$, \qquad Kangqun Zhang\(^2\)
\footnote{Daoyin He (101012711@seu.edu.cn) is supported by BK20233002 and 2242023R40009. Yaqing Sun (yaqingsun@njit.edu.cn) and Kangqun Zhang (chkq@njit.edu.cn) are supported by 23KJB110012, Yaqing Sun is also supported by YKJ202218. Daoyin He is the corresponding author.
}\vspace{0.5cm}
\\
\small 1.
School of Mathematics, Southeast University, Nanjing 211189, China.\\
\small 2.
School of Mathematical and Physics, Nanjing Institute of Technology,
Nanjing 210067, China.}

\vspace{-1.5cm}

\date{}
\maketitle
\thispagestyle{empty}

\centerline{}

\begin{abstract}
	In this paper we prove a sharp global existence result for semilinear wave equations with time-dependent scale-invariant damping terms if the initial data is small.

More specifically, we consider Cauchy problem of $\p_t^2u-\Delta u+\f{\mu}{t}\p_tu=|u|^p$, where $n\ge 3$, $t\ge 1$ and $\mu\in(0,1)\cup(1,2)$. For critical exponent \(p_{\crit}(n,\mu)\) which is the positive root of \((n+\mu-1)p^2-(n+\mu+1)p-2=0\) and conformal exponent \(p_{\conf}(n,\mu)=\frac{n+\mu+3}{n+\mu-1}\), we establish global existence for \(n\geq3\) and \(p_{\crit}(n,\mu)<p\leq p_{\conf}(n,\mu)\). The proof is based on changing the wave equation
 into the semilinear generalized Tricomi equation $\p_t^2u-t^m\Delta u=t^{\al(m)}|u|^p$,
 where $m=m(\mu)>0$ and  $\al(m)\in\Bbb R$ are two suitable constants, then we investigate
 more general semilinear Tricomi equation $\p_t^2v-t^m\Delta v=t^{\al}|v|^p$ and establish related weighted Strichartz estimates.
Returning to the original wave equation, the corresponding global existence results on the small data solution $u$ can be obtained.
\end{abstract}

\vskip 0.2 true cm

{\bf Keywords:} Global existence, Generalized Tricomi equation, Weighted Strichartz estimate, Fourier integral operator,
 Weak solution \vskip 0.2 true cm

{\bf Mathematical Subject Classification 2000:} 35L70, 35L65,
35L67
\tableofcontents

\vskip 0.5 true cm
\phantomsection
\section{Introduction}

\subsection{Setting of the problem and statement of the main result}
Consider the Cauchy problem of the semilinear wave equation with time-dependent damping
\begin{equation}\label{equ:eff}
\left\{ \enspace
\begin{aligned}
&\partial_t^2 u-\Delta u +\f{\mu}{t^{\beta }}\,\p_tu=|u|^p, &&
  (t,x)\in (t_0,\infty)\times \R^{n},\\
&u(t_0,x)=u_0(x), \quad \partial_{t} u(t_0,x)=u_1(x),
\end{aligned}
\right.
\end{equation}
where $\mu\ge 0$, $\beta\ge 0$, $p>1$, $n\ge 2$ and \(t_0\geq0\). When $\mu=0$ and \(t_0=0\) is , \eqref{equ:eff} becomes
\begin{equation}\label{equ:eff-1}
\left\{ \enspace
\begin{aligned}
&\partial_t^2 u-\Delta u=|u|^p, \\
&u(0,x)=u_0(x), \quad \partial_{t} u(0,x)=u_1(x).
\end{aligned}
\right.
\end{equation}
For \eqref{equ:eff-1}, Strauss in \cite{Strauss} proposed the following well-known conjecture (Strauss' conjecture):

{\it Let
$p_s(n)$ denote the positive root of the quadratic algebraic equation
\begin{equation}\label{1.3}
\left(n-1\right)p^2-\left(n+1\right)p-2=0.
\end{equation}
If $p>p_s(n)$, \eqref{equ:eff-1} admits a unique global solution with small data; whereas if $1<p\leq p_s(n)$, the solution of \eqref{equ:eff-1} can blow up in finite time.}

So far the Strauss' conjecture has been systematically studied and solved well, see
\cite{Gls1,Gla1, Gla2,Joh,Gls2,Sch,Sid,Yor}.
Especially, in \cite{Wang}, one can find the
detailed history on the studies of \eqref{equ:eff-1}.

When $\beta=0$ holds and $\mu=1$ is assumed without loss of generality,  \eqref{equ:eff} with \(t_0=0\) becomes
\begin{equation}\label{equ:eff-2}
\left\{ \enspace
\begin{aligned}
&\partial_t^2 u-\Delta u+\p_t u=|u|^p, \\
&u(0,x)=u_0(x), \quad \partial_{t} u(0,x)=u_1(x).
\end{aligned}
\right.
\end{equation}
Nowadays the study of \eqref{equ:eff-2} is almost classic, it is shown that
the solution $u$ can blow up in finite time when \(1<p\le p_f(n)\) with the Fujita exponent
$p_f(n)=1+\f{2}{n}$ as defined in \cite{Fuj} for the semilinear parabolic equations, while the small data solution $u$ exists globally
for $p>p_f(n)$ with $n=1,2$ and $p_f(n)<p<\f{n}{n-2}$ for $n\ge 3$, see
\cite{EGR}, \cite{Li}, \cite{TY},  \cite{W1} and \cite{ZQ}. Recently, Chen and Reissig in \cite{C-R} considered \eqref{equ:eff-2} with data from Sobolev spaces of negative order, they obtained a new critical exponent \(p_c=1+\frac{4}{n+2\gamma }\) and established small data global existence result for \(1<p<p_c\) and blow up result for \(1<p<p_c\) provided that \(1\leq n\leq6\), while the global existence for \(p=p_c\) is proved by D'Abbicco \cite{DA-1} in Euclidean setting and in the Heisenberg space.

For $\beta>1$ or $0<\beta<1$, consider the corresponding linear homogeneous equation
\begin{equation}\label{equ:linear}
	\left\{ \enspace
\begin{aligned}
&\partial_t^2 v-\Delta v +\f{\mu}{t^{\beta }}\,\p_tv=0, &&
  (t,x)\in (1,\infty)\times \R^{n},\\
&v(1,x)=v_0(x), \quad \partial_{t} v(1,x)=v_1(x),
\end{aligned}
\right.
\end{equation}
then the long time decay rate of $v$ just corresponds to that of linear wave equation or linear parabolic equation, respectively, see Wirth \cite{Wirth-1, Wirth-2}.
When $0<\beta <1$ and $\mu>0$, if $p>p_f(n)$ with \(n=1,2\) or \(p_f(n)<p<\frac{n+2}{n-2}\) with \(n\geq3\), then $\eqref{equ:eff}$ has a global small data solution $u$, see D'Abbicco, Luencte and Reissig \cite{Rei2}, Lin, Nishihara and Zhai \cite{Zhai} and Nishihara \cite{N};
if $1<p\le p_f(n)$, the solution $u$ generally blows up in finite time, one can be referred to Fujikawa, Ikeda and Wakasugi \cite{FIW}.
When $\beta>1$ and $\mu>0$,  the global existence or blowup results on \eqref{equ:eff} are analogous to the ones on
  \eqref{equ:eff-1}. It is shown that for \(p>p_s(n)\), the global small data solution exists (see Liu and Wang  \cite{LW}), while the
  solution may blow up in finite time for $1<p\le p_s(n)$ (see Lai and Takamura \cite{LT} and Wakasa, Yordanov \cite{WY}).

When $\beta=1$, the equation in \eqref{equ:linear} is invariant under the scaling
\[\tilde{u}(x,t):=u(\sigma x, \sigma t), \quad \sigma>0, \]
and hence we say the damping term is scale-invariant. In this case,
the behavior of \eqref{equ:eff-1} depends on the scale of \(\mu\). If \(n=1\) with \(\mu\geq\frac{5}{3}\) or \(n=2\) with \(\mu\geq3\) or \(n\geq3\) with \(\mu\geq n+2\) and \(p_f(n)<p<\frac{n}{n-2}\), then the global
existence of small data solution $u$ has been established in the important work \cite{DA} by D'Abbicco. For relatively small \(\mu\),
so far there are few results on  the global solution of \eqref{equ:eff} with $\beta=1$, recently, for \(n=3\) and the radial symmetrical case, the global small solution is shown in Lai and Zhou \cite{LZ} for \(\mu\in[\frac{3}{2},2)\) and \(p_s(3+\mu)<p\leq2\). On the other hand, there have been systematic results for blow up: for the case of \(n=1\), if \(0<\mu<\frac{4}{3}\) and \(1<p\leq p_s(1+\mu)\), or \(\mu\geq\frac{4}{3}\) and \(1<p\leq p_f(1)\), the solution will blowup in finite time; for the case of \(n\geq2\), the blowup results are also established when \(\mu>0\) and \(1<p<p_s(n+\mu)\), or \(0<\mu<\frac{n^2+n+2}{n+2}\) and \(1<p\le p_s(n+\mu)\) (see the works in \cite{IS, LTW, TL1, TL2, W1, W2} by Ikeda, Sobajima, Lai, Takamura, Wakasa, Tu, Lin and Wakasugi).
These results enlighten us the critical exponent for \(0<\mu<\frac{n^2+n+2}{n+2}\) should be \(p_{\crit}(n,\mu)=p_s(n+\mu)\) where \(p_s(n+\mu)\) is the positive root of the following quadratic equation:
\begin{equation}\label{equ:crit-mu}
	(n+\mu-1)p^2-(n+\mu+1)p-2=0.
\end{equation}

From our last paper \cite{HLYin}, we start to study the global existence of small data solutions to problem \eqref{equ:eff} with $\beta=1$,
$\mu\in (0,1)\cup(1,2)$ and $p>p_s(n+\mu)$.
That is, we focus on the following regular Cauchy problem
\begin{equation}\label{equ:eff1}
\left\{ \enspace
\begin{aligned}
&\partial_t^2 u-\Delta u +\f{\mu}{t}\,\p_tu=|u|^p, \\
&u(1,x)=\varepsilon u_0(x), \quad \partial_{t} u(1,x)=\varepsilon u_1(x),
\end{aligned}
\right.
\end{equation}
where \((u_0, u_1)\in C^\infty_c(\R^n)\) and $n\ge 2$. Comparing \eqref{1.3} and \eqref{equ:crit-mu}, we see that we have a shift from \(n\) to \(n+\mu\) in the coefficients of the quadratic equations of \(p\). Thus the conformal exponent for \eqref{equ:eff-1} should be
\begin{equation}
	p_{\conf}(n,\mu)=\frac{n+\mu+3}{n+\mu-1}.
\end{equation}
In \cite{HLYin}, we have obtained global existence result for small data solution for \(n\geq2\), \(\mu\in(0,1)\) and \(p\geq\max\{p_{\conf}(m,\mu), 3\}\), or \(\mu\in(1,2)\) and \(p\geq\max\{p_{\conf}(n,\mu), \frac{4}{\mu}-1 \}\).
The main results in this article can be stated as follows.

\begin{theorem}\label{YH-1}
Assume that \(n\geq3\), $\mu\in(0,1)\cup(1,2)$ and \(p_{crit}(n,\mu)<p\leq p_{conf}(n,\mu)\).
Suppose $u_i\in C_c^{\infty}(\R^n)$ ($i=0, 1$),  then for $\ve>0$ small enough problem \eqref{equ:eff1} admits a global weak solution \(u\) with
\begin{equation}	\left(1+\big|t^2-|x|^2\big|\right)^{\gamma}t^{\frac{\mu}{p+1}}u\in L^{p+1}([1,\infty)\times\R^{n}),
\end{equation}
for some \(\gamma\) satifying
\[\frac{1}{p(p+1)}<\gamma<\frac{(n+\mu-1)p-(n+\mu+1)}{2(p+1)}. \]
\end{theorem}
\begin{remark}\label{JY-4}
{\it For $\mu=1$, under the Liouville transformation $v=(1+t)^{\f12}u$, the equation in \eqref{equ:eff1}
becomes the semilinear Klein-Gordon equation $\p_t^2v-\Delta v+\f{1}{4(1+t)^2}v=(1+t)^{\f{1-p}{2}}|v|^p$
with time-dependent coefficient. By our knowledge, so far there are few results on the global existence of $u$
in \eqref{equ:eff1} with $\mu=1$.}
\end{remark}

\begin{remark}\label{JY-2}
{\it For $\mu=2$, set $v=(1+t)u$, then the equation in \eqref{equ:eff1}
can be changed into the undamped wave equation $\p_t^2v-\Delta v=(1+t)^{1-p}|v|^p$.
In this case, there have been a lot of results about global existence for the small data solutions, see D'Abbicco and Lucente \cite{DL}, and  D'Abbicco, Lucente and Reissig \cite{Rei1}, Kato and Sakuraba \cite{Kato} and Palmieri \cite{P}. For examples, when \(1\leq n\leq3\), the critical indices have been
determined as \(\max\{p_f(n), p_s(n+2)\}\);
when \(n\geq4\) and \(n\) is even, the global existence is established for \(p_s(n+2)<p<p_f(\frac{n+1}{2})\)
provided that the solution is radial symmetric; the global existence results also hold for \(n\geq5\) when \(n\) is odd, \(p_s(n+2)<p<\min\{2,\frac{n+3}{n-1}\}\) and the solution is radial symmetric. Recently, we obtain some
global existence results for the general small data solution of \eqref{equ:eff1}
when \(\mu=2\) and \(p\) is suitably large in Theorems 1.3 of our former work \cite{HLYin} .}
\end{remark}

\begin{remark}
	For \(n=1\), it is proven recently by Li and Yin \cite{LY-1} that for \(\mu\in(0,1)\cup(1,\frac{4}{3})\), the solution of \eqref{equ:eff1} exists for small data. While for \(n=2\), Li and Yin establish the small data global existence for
	\(\mu\in(0,1)\cup(1,2)\) in \cite{LY-2}.
\end{remark}
Note that  when  $0<\mu<1$, as in D'Abbicco, Lucente and Reissig \cite{Rei1}, if setting $\mu=\f{k}{k+1}$ with $k\in (0,\infty)$
and $t=T^{k+1}/(k+1)$, then the equation in \eqref{equ:eff1} is essentially equivalent to
\begin{equation}\label{YC-1}
\partial_T^2 u-T^{2k}\Delta u=T^{2k} |u|^p;
\end{equation}
when $1<\mu<2$, if setting
\[v(t,x)=t^{\mu-1}u(t,x),\]
then the equation in \eqref{equ:eff1} can be written as
\begin{equation}\label{equ:shift1-0}
	\partial_t^2 v-\Delta v +\f{\t\mu}{t}\,\p_tv=t^{(p-1)(\t\mu-1)}|v|^p,
\end{equation}
where $\t\mu=2-\mu\in (0,1)$, and the unimportant constant coefficients $C_{\mu}>0$ before the nonlinearities in
\eqref{YC-1} and \eqref{equ:shift1-0} are neglected. Let $\t\mu=\f{\t k}{\t k+1}$ with $\t k\in (0,\infty)$
and $t=T^{\t k+1}/(\t k+1)$. Then for $t\ge 1$, \eqref{equ:shift1-0} is actually equivalent to
\begin{equation}\label{YC-2}
\partial_T^2 u-T^{2\t k}\Delta u=T^{\al} |u|^p,
\end{equation}
where $\al=2\t k+1-p$, and the unimportant constant coefficient $C_{\t k}>0$ before the nonlinearity
of \eqref{YC-2} is also neglected.
Based on \eqref{YC-1} and \eqref{YC-2}, in order to prove Theorems \ref{YH-1},
it is necessary to study the following semilinear generalized Tricomi equation
\begin{align}\label{YH-3}
\partial_t^2 u-t^m \Delta u =t^\alpha|u|^p,
\end{align}
where $m>0$ and $\al\in\Bbb R$.

The local existence result for \eqref{YH-3} was first considered by the pioneering work \cite{Yag2} of Yagdjian. Under the conditions of \(n\geq2\), \(\alpha>-1\) and
\begin{equation}\label{equ:blow}
1<p<1+\f{4}{(m+2)n-2},
\end{equation}
it is proved in \cite[Theorem~1.3]{Yag2} that \eqref{YH-3} addmits no global solution $u\in C([0, \infty), L^{p+1}(\R^n))$.
On the other hand, if \(n\geq2\), \(\alpha>-1\) and
\begin{equation}\label{equ:ugly1}
  \left\{ \enspace
\begin{aligned}
	p&\leq1+\frac{2m}{(m+2)n+2}, \\
		\frac{(m+2)np}{p+1}&\geq\frac{2(p+1)+\max\{\alpha,\alpha p \}}{p-1} .
\end{aligned}
\right.	
  \end{equation}
then \cite[Theorem~1.2]{Yag2}
that \eqref{YH-3} has a global small data solution
$u\in C([0, \infty), L^{p+1}(\R^n)) \cap C^1([0, \infty), {\mathcal D}'(\R^n))$.
By checking \eqref{equ:blow} and \eqref{equ:ugly1} one see that the results in \cite{Yag2} were not completed
since there is a gap between the intervals for \(p\) where blowup happens or small data solution exists globally, also an upper bound of \(p\) exists for the global existence part. Later Galstian in \cite{Gal} investigated Cauchy problem for \eqref{YH-3} for the case \(n=1\), \(\alpha>-1\) and \(m\geq0\). She proved the blowup result for \(1<p<1+\frac{4}{m}\) and established the global existence for small data solution under the condition \eqref{equ:ugly1}.

Witt, Yin and the first author of this article considered the case $\al=0$ and $m>0$,
and for the initial data problem of \eqref{YH-3} starting from some positive time $t_0$,
it has been shown that there exists a critical index \(p_{\crit}(n,m)>1\) such that
when $p>p_{\crit}(n,m)$, the small data solution $u$ of \eqref{YH-3} exists globally;
when $1<p\le p_{\crit}(n,m)$, the solution $u$ may blow up in finite time,
where \(p_{\crit}(n,m)\) for \(n\geq2\) and \(m>0\) is the positive root of
\begin{equation}\label{equ:p2}
\Big((m+2)\frac{n}{2}-1\Big)p^2+\Big((m+2)(1-\frac{n}{2})-3\Big)p
-(m+2)=0,
\end{equation}
while for \(n=1\), \(p_{\crit}(1,m)=1+\frac{4}{m} \) (see
\cite{HWY1,HWY2,HWY3,HWY4,HWY5}).

It is pointed out that  about the local existence and regularity of solution~$u$ to
\eqref{YH-3} with $\al=0$ and $m\in\Bbb N$ under the weak regularity assumptions of initial data
$(u,\p_tu)(0,x)=(u_0,
u_1)$, the reader may consult Ruan, Witt and Yin \cite{Rua1, Rua2, Rua3, Rua4}.

Now let us turn back to the problem \eqref{YH-3}. Here and below denote
\[\phi_m(t)=\f{2}{m+2}t^{\f{m+2}{2}},\quad \text{and} \quad T_0=\phi_m^{-1}(1)=\left(\frac{m+2}{2} \right)^{\frac{2}{m+2}},\]
note that \(T_0\in(1,e^{1/e}]\) for all \(m>0\). We next focus on the global existence of the solution to the following problem
\begin{equation}\label{YH-4}
\begin{cases}\partial_t^2 u-t^{m} \Delta u=t^\alpha|u|^p, \\
u(T_0, x)=\ve f(x), \p_t u(T_0, x)=\ve
g(x),
\end{cases}
\end{equation}
where $m>0$, $\alpha\in\Bbb R$, $p>1$, $f(x), g(x)\in C_c^{\infty}(\Bbb R^n)$ with $n\ge 2$ and
supp $f$, supp $g\in B(0,M)$ with $M>0$.
It has been shown in Theorem 1 of Palmieri and Reissig \cite{PR} that there exists a critical exponent $p_{crit}(n,m,\alpha)>1$ for $\al>-2$
such that when $1<p\le p_{crit}(n,m,\alpha)$, the solution of \eqref{YH-4} can blow up in finite time with some suitable
choices of $(u_0,u_1)$, where
\begin{equation}\label{YH-5}
p_{crit}(n,m,\alpha)=\max \left\{p_1(n,m,\alpha), p_2(n,m,\alpha)\right\}
\end{equation}
with $p_1(n,m,\alpha)=1+\frac{2(2+\alpha)}{(m+2)n-2}$ and $p_2(n,m,\alpha)$ being the positive root of the quadratic equation:
\begin{equation}\label{YH-7}
\left(\frac{m}{2(m+2)}+\frac{n-1}{2}\right) p^2-\left(\frac{n+1}{2}-\frac{3 m}{2(m+2)}+\frac{2 \alpha}{m+2}\right) p-1=0.
\end{equation}
It is not difficult to verify that when $n\ge 3$ and $\al>-2$ or $n=2$ and $\al>-1$, $p_{crit}(n,m,\alpha)=p_2(n,m,\alpha)$
holds; when $n=2$ and $-2<\al\le -1$, $p_{crit}(n,m,\alpha)=p_1(n,m,\alpha)$ holds. However, as we will see in next subsection, in order to prove Theorem~\ref{YH-1},
it suffices to consider \(-1<\alpha\leq m\) for \(n\geq2\) in \eqref{YH-4} , thus we set  $p_{crit}(n,m,\alpha)=p_2(n,m,\alpha)$ in this article.

In \cite[(1.24)]{HLYin},  we have determined a conformal exponent
\begin{equation}\label{equ:conformal}
	p_{\conf}(n,m,\alpha)=\frac{(m+2)n+4\alpha+6}{(m+2)n-2}.
\end{equation}
Then we can state the global existence result for \eqref{YH-4}
\begin{theorem}\label{thm:global existence}{\bf (Global existence for semilinear Tricomi equation)}
Let \(m\in(0,\infty)\), assume that \(-1<\alpha\leq m\) for \(n\geq2\). For $p\in \big(p_{crit}(n,m,\alpha ), p_{conf}(n,m,\alpha)\big]$, there exists a constant \(\varepsilon_0>0\) such that if \(0<\ve<\ve_0\), then problem \eqref{YH-4} has a  global weak solution $u$ with
\begin{equation}\label{equ:1.3}
\left(1+\big|\phi_m^2(t)-|x|^2\big|\right)^{\gamma}t^{\frac{\alpha}{p+1}}u\in L^{p+1}([T_0,\infty)\times \mathbb{R}^n),
\end{equation}
for some positive constant $\gamma$ fulfills
\begin{equation}\label{equ:1.4}
\f{1}{p(p+1)}<\gamma<\frac{n}{2}-\frac{1}{m+2}-\frac{1}{p+1}\lc n+\frac{2\alpha-m}{m+2}\rc.
\end{equation}
\end{theorem}

\begin{remark}
	Combing Theorem 1.2 together with the blowup result in Palmieri and Reissig \cite[Theorem 1.1]{PR}, the blowup vs global existence problem for the regular Cauchy problem of the semilinear Tricomi equation \eqref{YH-4} has been solved for \(n\geq2\).
\end{remark}

\subsection{Proof of the main theorem}
Take the results in Theorem~\ref{thm:global existence} as granted, we can prove the main theorem.
\paragraph{Case I \(\mathbf{ 0<\mu<1}\)}
For any fixed \(\mu\in(0,1)\), there exists unique \(m\in(0,\infty)\) such that $\mu=\f{m}{m+2}$.
Denote $\f{2}{m+2}t^{\f{m+2}{2}}$ as $\tilde{t}$, then for $t\ge T_0$,
\eqref{YH-4} is equivalent to
\begin{equation}\label{equ:mp}
\left\{ \enspace
\begin{aligned}
&\p_{\tilde{t}}^2 u-\Delta u+\f{m}{(m+2)\tilde{t}}\,\p_{\tilde{t}}u=\tilde{t}^{\frac{2(\alpha-m)}{m+2}}|u|^p, \\
&u(1,x)=u_0(x), \quad \partial_{\tilde{t}} u(1,x)=u_1(x),
\end{aligned}
\right.
\end{equation}
Then by choosing $\al=m$, we see that \eqref{equ:mp} is equivalent to \eqref{equ:eff1} with \(\mu=\frac{m}{m+2}\).
The conditon $\alpha=m=\f{2\mu}{1-\mu}$ implies that \eqref{YH-7} is equivalent to
\[(n+\mu-1)p^2-(n+\mu+1)p-2=0,\]
thus \(p_{\crit}(n,m,\alpha)=p_{\crit}(n,\mu)\). In addition,
\begin{equation*}	p_{\conf}(n,m,\alpha)=\frac{(\f{2\mu}{1-\mu}+2)n+4\f{2\mu}{1-\mu}+6}{(\f{2\mu}{1-\mu}+2)n-2}=\frac{n+\mu+3}{n+\mu-1}=p_{\conf}(n,\mu).
\end{equation*}
Thus Theorem~\ref{YH-1} follows immediately from Theorem~\ref{thm:global existence} provided \(\mu\in(0,1)\).

\paragraph{Case II \(\mathbf{ 1<\mu<2}\)}
For \(n\geq3\) and \(\mu\in(1,2)\), we intend to establish global existence for \eqref{YH-4} for each \(p\in(p_{\crit}(n,\mu), p_{\conf}(n,\mu)]\).  To this aim, let
\[v(t,x)=(1+t)^{\mu-1}u(t,x).\]
Then the equation in \eqref{equ:eff1} can be written as
\begin{equation}\label{equ:shift1}
	\partial_t^2 v-\Delta v +\f{2-\mu}{1+t}\,\p_tv=(1+t)^{(p-1)(1-\mu)}|v|^p.
\end{equation}
Comparing \eqref{equ:shift1} with \eqref{equ:mp}, one can see that the global existence result of \eqref{equ:eff1}
can be derived from \eqref{equ:mp} for \(\mu=2-\frac{m}{m+2}\) and \(\frac{2(\alpha-m)}{m+2}=(p-1)(1-\mu)\).
This leads to the following choice of \(\alpha\) for $\mu\in(1,2)$,
\begin{equation}\label{equ:lbp1}
	\alpha=1+m-p.
\end{equation}
Hence in order to prove Theorem 1.1 for \(\mu\in(1,2)\), we must guarantee the global existence result in Theorem 1.2 is applicable to \(m=\frac{2(2-\mu)}{\mu-1}\) and every \(\alpha\) satisfying
\[1+m-p_{\conf}(n,\mu)\leq\alpha<1+m-p_{\crit}(n,\mu).\]
Since \(p_{\crit}(n,\mu)>1\), we have \(1+m-p_{\crit}(n,\mu)<m\). On the other hand, by \(\mu=2-\frac{m}{m+2}\),
\begin{equation}
	p_{\conf}(n,\mu)=1+\frac{4}{n+\mu-1}=1+\frac{4(m+2)}{(m+2)n+2}.
\end{equation}
Thus \(1+m-p_{\conf}(n,\mu)=m-\frac{4(m+2)}{(m+2)n+2}\), a direct computation shows that
\(m-\frac{4(m+2)}{(m+2)n+2}>-1\) for \(n\geq3\), therefore
\[\big[1+m-p_{\conf}(n,\mu),1+m-p_{\crit}(n,\mu)\big)\subseteq(-1,m), \quad\text{for}\quad n\geq3, \]
and Theorem~\ref{thm:global existence} implies Theorem~\ref{YH-1} for \(\mu\in(1,2)\).
\subsection{Some remarks}
\begin{remark}
	For the case \(n=2\), if \(\mu\in(0,1)\cup(1, 1+\frac{2}{\sqrt{2}+1})\), the method in Section 1.2 can also imply global existence for \(p_{crit}(2,\mu)<p< p_{conf}(2,\mu)\). However, if \(1+\frac{2}{\sqrt{2}+1}\leq\mu<2\) and \(\frac{2}{\mu-1}\leq p<p_{\conf}(2,\mu)\), the global existence for \eqref{equ:eff-1} is related the global existence for \eqref{YH-4} with \(-\frac{4}{3}<\alpha\leq-1\). Unfortunately, in this range of \(\alpha\), we have \(p_{\crit}(n,m,\alpha)=p_1(n,m,\alpha)=1+\frac{2(2+\alpha)}{(m+2)n-2}\) and the homogeneous Strichartz estimate Lemma 2.1 is no longer valid and the method of this article is not applicable. Recently, Yin and Li in \cite{LY-2} solved this gap by establishing angular mixed-norm Strichartz-type equation for semilinear Tricomi equation, then the small global existence for \eqref{equ:eff-1} with \(n=2\) and \(\mu\in(0,1)\cup(1,2)\) is obtained in \cite{LY-2}.
\end{remark}
\begin{remark}\label{JY-3}
{\it For $\mu=1$, the equation in \eqref{equ:eff-1} can not be transformed to Tricomi equation, and we will treat \(\mu=1\) in our future work with different method.}
\end{remark}

\subsection{Sketch for the proof of Theorem 1.2}
We now sketch some of the ideas in the proof of Theorem 1.2.
To prove Theorem 1.2, motivated by \cite{Gls1, Gls2}, where some weighted Strichartz estimates with the characteristical weight $1+|t^2-|x|^2|$ were obtained for the linear wave operator
$\p_t^2-\triangle$, we need to establish some Strichartz estimates with the weight
$\big((\phi_m(t)+M)^2-|x|^2\big)^\gamma t^\beta $  for the generalized
Tricomi operator $\p_t^2-t^m\Delta$ with suitable index \(\beta\) and \(\gamma\). As we will see in Lemma 2.1, it is not difficult to get the linear homogeneous Strichatz estimate once we have the pointwise decay estiamtes.

In next step, we turn to establish weighted Strichartz estimates for the linear inhomogeneous problem:
\begin{equation}
\begin{cases}
&\partial_t^2 w-t^m\triangle w=F(t,x), \\
&w(T_0,x)=0,\quad \partial_tw(T_0,x)=0,
\end{cases}
\label{Y-3}
\end{equation}
which is the most difficult part in this article. The main result is the following:

\begin{theorem}\label{thm:inhomogeneous estimate}
Let \(n\geq2\), \(m>0\) and \(-1<\alpha\leq m\). For problem \eqref{Y-3}, if $F(t,x)\equiv0$ when
$|x|>\phi_m(t)-1$, then there exist some constants $\gamma_1$ and $\gamma_2$ satisfying $\gamma_1
<\frac{n}{2}-\frac{1}{m+2}-\frac{1}{q}( n+\frac{2\alpha-m}{m+2})$
and $\gamma_2>\frac{1}{q}$, such that
\begin{equation}
\begin{split}
\lp\big(\phi^2_m(t)-|x|^2\big)^{\gamma_1}t^{\frac{\alpha}{q}}w\rp_{L^q([T_0, \infty)\times \mathbb{R}^{n}))}\leq C
\lp\big(\phi^2_m(t)-|x|^2\big)^{\gamma_2}t^{-\frac{\alpha}{q}}F\rp_{L^{\frac{q}{q-1}}([T_0, \infty)\times \mathbb{R}^{n}))},
\end{split}
\label{equ:3.2}
\end{equation}
where $2\leq q\leq\f{2((m+2)n+2+2\alpha )}{(m+2)n-2}$,
and $C>0$ is a constant  depending on $n$, $m$, \(\alpha\), $q$, $\g_1$ and $\gamma_2$.
\end{theorem}
Based on Theorem~\ref{thm:inhomogeneous estimate}, repeating the proof of Theorem 1.4 in \cite{HWY4} we can prove the following modified version of \eqref{equ:3.2}.
\begin{theorem}\label{thm:useful estimate}
Let \(n\geq2\), \(m>0\) and \(-1<\alpha\leq m\). For problem \eqref{Y-3}, if $F(t,x)\equiv0$ when
$|x|>\phi_m(t)+M-1$, then for \(\gamma_1\) and \(\gamma_2\) in Theorem 1.3,

\begin{equation}
\begin{split}
\Big\|\Big(\big(\phi_m(t)+M\big)^2-|x|^2\Big)^{\gamma_1}&t^{\frac{\alpha}{q}}w\Big\|_{L^q([T_0, \infty)\times \mathbb{R}^{n})} \\
& \leq C
\Big\|\Big(\big(\phi_m(t)+M\big)^2-|x|^2\Big)^{\gamma_2}t^{-\frac{\alpha}{q}}F\Big\|_{L^{\frac{q}{q-1}}([T_0, \infty)\times \mathbb{R}^{n})},
\end{split}
\label{equ:3.5}
\end{equation}
where $2\leq q\leq\f{2((m+2)n+2+2\alpha )}{(m+2)n-2}$, and $C>0$ is a constant  depending on $n$, $m$, \(\alpha\), $q$, $\g_1$, $\gamma_2$ and \(M\).
\end{theorem}

With Lemma 2.1 and Theorem~\ref{thm:useful estimate}, we can prove Theorem~\ref{thm:global existence} by a standard Picard iteration, the detail is given in Section 5.

It remains to give the proof of Theorem~\ref{thm:inhomogeneous estimate}. By Stein's analytic interpolation theorem (see \cite{Stein}), in order to prove \eqref{equ:3.2},
it suffices to establish \eqref{equ:3.2} for the two extreme cases of \(q=q_0=\frac{2((m+2)n+2+2\alpha )}{(m+2)n-2}\) and \(q=2\):
\begin{equation}
\begin{split}
\lp\big(\phi_m(t)^2-|x|^2\big)^{\gamma_1}t^{\frac{\alpha}{q_0}}w\rp_{L^{q_0}(\ra)}\leq C
\lp\big(\phi_m(t)^2-|x|^2\big)^{\gamma_2}t^{-\frac{\alpha}{q_0}}F\rp_{L^{\frac{q_0}{q_0-1}}(\ra)},
\end{split}
\label{equ:3.3}
\end{equation}
where $\gamma_1<\frac{1}{q_0}<\gamma_2$; and
\begin{equation}
\begin{split}
\lp\big(\phi_m(t)^2-|x|^2\big)^{\gamma_1}t^{\frac{\alpha}{2}}w\rp_{L^2(\ra)}\leq C
\lp\big(\phi_m(t)^2-|x|^2\big)^{\gamma_2}t^{-\frac{\alpha}{2}}F\rp_{L^2(\ra)}, \\
\end{split}
\label{equ:3.4.1}
\end{equation}
where $\gamma_1<-\frac{1}{2}+\frac{m-\alpha}{m+2}$ and $\gamma_2>\frac{1}{2}$.

In order to derive \eqref{equ:3.3}, we use the idea from \cite{Gls1} and split the integral
domain $\{(t,x): \phi^2_m(t)-|x|^2\leq1\}$ into some pieces, which correspond to the ``relatively small times" part and the ``relatively large times"
part respectively.

For the case of ``relatively small times", \textbf{the key point is to establish the inhomogeneous Strichartz estimates without characteristic weight at the endpoint} \(\mathbf{q=q_0}\) \textbf{for all} \(\mathbf{n\geq2}\), \(\mathbf{m>0}\) \textbf{and}
 \(\mathbf{0<\alpha\leq\frac{n}{n-1}\cdot m}\) :
\begin{equation}\label{equ:un-chara}
	\mathbf{\lp t^{\frac{\alpha}{q_0}}w\rp_{L^{q_0}(\ra)}\leq C
\lp t^{-\frac{\alpha}{q_0}}F\rp_{L^{\frac{q_0}{q_0-1}}(\ra)}}
\end{equation}
\textbf{\eqref{equ:un-chara} is an essential improvement of Lemma 2.2 in \cite{HLYin} (see Remark 2.1 for explanation in details), it also make the proof in this article more self-contained.} A local in time version of \eqref{equ:un-chara}
can also be established for all \(n\geq2\), \(m>0\) and \(-1<\alpha\leq0\) :
\begin{equation}\label{equ:un-chara-1}
	\lp t^{\frac{\alpha}{q_0}}w\rp_{L^{q_0}([T_0, \bar{T}]\times\R^n )}\leq C
\lp t^{-\frac{\alpha}{q_0}}F\rp_{L^{\frac{q_0}{q_0-1}}([T_0, \bar{T}]\times\R^n)},
\end{equation}
for any fixed large \(\bar{T}\). Application of \eqref{equ:un-chara} and \eqref{equ:un-chara-1} together with techniques of convolution type inequalities gives \eqref{equ:3.3}.

For the case of ``relatively small times", the analysis is more technical and involved. We have to divide the integral domain according to the scale of \(\phi_m(t)-|x|\), if the \(|\phi_m(t)-|x||\) is very small or very large, then \eqref{equ:3.3} follows by delicate analysis of support condition for \(w\) and \(F\). While for \(|\phi_m(t)-|x||\) with medium scale, we introduce such kinds of Fourier integral operators for $z\in\Bbb C$,
\begin{align*}
\begin{split}
(\mathcal{T}_zg)(t,x)=&\left(z-\frac{(m+2)n+2+2\alpha}{(m+2)(\alpha+2)}\right)e^{z^2}t^{\frac{\alpha}{q_0}}  \\
&\times \int_{\mathbb{R}^n}\int_{\mathbb{R}^n}
e^{i\{(x-y)\cdot\xi-[\phi_m(t)-|y|]|\xi|\}}\big(1+\phi_m(t)|\xi|\big)^{-\frac{m}{2(m+2)}}g(y)\frac{\md\xi}{|\xi|^z}\md y, \\
(\tilde{\mathcal{T}}_zg)(t,x)=&\left(z-\frac{(m+2)n+2+2\alpha}{2(\alpha +2)}\right)e^{z^2}
t^{\frac{\alpha}{q_0}}\phi_m(t)^{-\frac{m}{2(m+2)}}  \\
&\times \int_{\mathbb{R}^n}\int_{\mathbb{R}^n}e^{i\{(x-y)\cdot\xi-[\phi_m(t)-|y|]|\xi|\}}g(y)\frac{\md\xi}{|\xi|^z}\md y,
\end{split}
\label{equ:4.58}
\end{align*}
applying these Fourier integral operators according to different range of \(\alpha\) respectively, and combining with the complex interpolation methods, we ultimately obtain \eqref{equ:3.3}.

For the \(L^2\) estimate \eqref{equ:3.4.1}, the idea is similar to that of \eqref{equ:3.3}, we split again the integral
domain $\{(t,x): \phi^2_m(t)-|x|^2\leq1\}$ in the corresponding Fourier integral operators into some pieces,
which correspond to the ``relatively small times" part and the ``relatively large times" part respectively. \textbf{The main innovation for the} \(L^2\) \textbf{estimate is in the ``relatively large times" part.  We apply new technique of inequality when both the scales of} \(\mathbf{\phi_m(t)+|x|}\) \textbf{and} \(\mathbf{\phi_m(t)-|x|}\) \textbf{are large, more specifically, we replace} \(\mathbf{(\phi_m(t)^2-|x|^2)^\beta}\) \textbf{with} \(\mathbf{(\phi_m(t)+|x|)^{\beta_1}(\phi_m(t)-|x|)^{\beta_2}}\) \textbf{and let} \(\mathbf{\beta_1}\) \textbf{and} \(\mathbf{\beta_2}\) \textbf{be different indices, this modification implies more decay rate and the} \(\mathbf{L^2}\) \textbf{estimate can be improved and we are able to handle all the dimension} \(\mathbf{n\geq2}\) \textbf{in the same manner. Recall that in our former work} \cite{HWY4}, \textbf{the} \(\mathbf{L^2}\) \textbf{estimate can be established only for} \(\mathbf{n\geq3}\), \textbf{hence} \eqref{equ:3.4.1} \textbf{is another essential improvement in this article.}

This paper is organized as follows: In Section $2$, we establish some basic estimates which includes the Strichartz inequality for the linear homogeneous equation $\partial_t^2 v-t^m\triangle v=0$ and \eqref{equ:un-chara}-\eqref{equ:un-chara-1}, and introduce some necessary results related to Littlewood-Paley decomposition. In Section $3$ and Section $4$, we will show the proofs of the endpoint inequalities \eqref{equ:3.3} and \eqref{equ:3.4.1}, respectively. The proof of Theorem 1.2 is then given in Section $5$.
In addition, some elementary but important estimates used in Section $3$ and Section $4$ are discussed further in the appendix.

\section{Basic Estimates}

In this section, our purpose is to establish some basic estimates and list some necessary results.
\subsection{Homogeneous Strichartz estimates with characteristic weight}
The first one is the homogeneous Strichartz estimate, for which we consider the following linear homogeneous problem:
\begin{equation}
\partial_t^2 v-t^m\triangle v=0, \qquad
v(T_0,x)=f(x),\quad \partial_tv(T_0,x)=g(x),
\label{Y-2}
\end{equation}
where $f, g\in C_c^\infty(\mathbb{R}^n)$, and $\operatorname{supp}\ (f,g)\subseteq \{x: |x|\leq M\}$. By  pointwise estimate in \cite{HWY4} and some delicate computation of \(L^q_{t,x}\) norm, we prove the following estimate.

\begin{lemma}
\label{lem:homogeneous estimate}
Let \(n\geq2\), \(m>0\) and \(\alpha>-1\). For the solution $v$ of \eqref{Y-2}, one then has
\begin{equation}\label{equ:2.2}
\begin{split}
	\left\|\Big(\big(\phi_m(t)+M\big)^2-|x|^2\Big)^\gamma t^{\frac{\alpha}{q}}v\right\|&_{L^q([T_0,+\infty)\times\mathbb{R}^n)} \\
	&\leq C\left(\| f\|_{W^{\f{n}{2}+\f{1}{m+2}+\delta,1}(\mathbb{R}^n)}+\| g\|_{W^{\f{n}{2}-\f{1}{m+2}+\delta,1}(\mathbb{R}^n)}\right),
\end{split}
\end{equation}
where
$p_{\crit}(n,m,\alpha)+1<q\leq p_{\conf}(n,m,\alpha)+1 $,
$0<\gamma<\f{(m+2)n-2}{2(m+2)}-\f{(m+2)n+2\alpha -m}{(m+2)q}$,
$\delta>0$ is small enoughy,
and $C$ is a positive constant depending on $n$, $m$, \(\alpha\), $q$, $\gamma$, $\delta$ and \(M\).
\end{lemma}

\begin{proof}
We first denote
\[A(f,g)=\left(\|f\|_{W^{\frac{n}{2}+\frac{1}{m+2}+\delta,1}(\mathbb{R}^n)}
+\|g\|_{W^{\frac{n}{2}-\frac{1}{m+2}+\delta,1}(\mathbb{R}^n)}\right),\]
then it follows from Section 2 of  \cite{HWY4} (see formula 2-20) that
the solution $v$ of \eqref{Y-2} satisfies
\[
\begin{split}
	|v|\leq & C_{m,n,\delta}(1+\phi_m(t))^{-\frac{n}{2}+\frac{1}{m+2}}(1+\big||x|-\phi_m(t)\big|)^{-\frac{n}{2}+\frac{1}{m+2}+\delta}A(f,g),
\end{split}
\]
where \(\delta>0\) is small enough. Since \(t\geq1\), we have for all \(\alpha>-1\)
\begin{equation}\label{equ:2.22}
\begin{split}
	|t^{\frac{\alpha}{q}}v|\leq & C_{m,n,\delta}\phi_m(t)^{-\frac{n}{2}+\frac{1}{m+2}+\frac{2\alpha}{q(m+2)}}(1+\big||x|-\phi_m(t)\big|)^{-\frac{n}{2}+\frac{1}{m+2}+\delta}A(f,g).
\end{split}
\end{equation}

Then we can compute the integral in the left hand side of \eqref{equ:2.2} by \eqref{equ:2.22} and the polar coordinate transformation.
then one can calculate
\begin{equation}\label{equ:2.23}
\begin{split}
&\Big\|\big((\phi_m(t)+M)^2-|x|^2\big)^\gamma t^{\frac{\alpha}{q}}v\Big\|_{L^q(\ra)}^q \\
&\le C_{m,n,\delta}A(f,g)\int_{T_0}^\infty\int_{|x|\leq \phi_m(t)+M}
\bigg\{\Big(\big(\phi_m(t)+M\big)^2-|x|^2\Big)^\gamma \\
&\qquad\qquad \qquad
\times\phi_m(t)^{-\frac{n}{2}+\frac{1}{m+2}+\frac{2\alpha}{q(m+2)}}\Big(1+\big||x|-\phi_m(t)\big|\Big)^{-\frac{n}{2}+\frac{1}{m+2}+\delta}\bigg\}^q dxdt
\\
&\leq C_{m,n,\delta}A(f,g)
\int_{T_0}^\infty\int_0^{\phi_m(t)+M}\Big\{\big(\phi_m(t)+M+r\big)^\gamma\big(\phi_m(t)+M-r\big)^\gamma \\
&\qquad\qquad\qquad\times\phi_m(t)^{-\frac{n}{2}+\frac{1}{m+2}+\frac{2\alpha}{q(m+2)}}
\big(1+|r-\phi_m(t)|\big)^{-\frac{n}{2}+\frac{1}{m+2}+\delta}\Big\}^q r^{n-1}drdt
\\
&\leq C_{m,n,\delta}A(f,g)
 \\
&\times\int_{T_0}^\infty\phi_m(t)^{q(-\frac{n}{2}+\frac{1}{m+2}
+\frac{2\alpha}{q(m+2)}+\gamma)}
\int_0^{\phi_m(t)+M}\big(1+|r-\phi_m(t)|\big)^{q(\gamma-\frac{n}{2}
+\frac{1}{m+2}+\delta)}r^{n-1}drdt.
\end{split}
\end{equation}

Notice that by our assumption, $\gamma-\frac{n}{2}+\frac{1}{m+2}+\frac{2\alpha}{q(m+2)}<(\frac{m}{m+2}-n)\frac{1}{q}$ holds.
Then we can choose a constant $\sigma>0$ such that
\[\gamma-\frac{n}{2}+\frac{1}{m+2}+\frac{2\alpha}{q(m+2)}<\Big(\frac{m}{m+2}-n\Big)\frac{1}{q}-\sigma.\]
Furthermore, we have
\[\Big(\gamma-\frac{n}{2}+\frac{1}{m+2}+\delta \Big)q<q\delta+ \frac{m-2\alpha}{m+2}-n.\]
In order to keep the integral in \eqref{equ:2.23} bounded, we need \(q\delta+\frac{m-2\alpha}{m+2}-n<-1\),
which is satisfied for \(\alpha>-1\), \(n\geq2\) and sufficiently small \(\delta\).
Then for some positive constant $\bar{\sigma}>0$, the integral in the last line of \eqref{equ:2.23} can be controlled by
\begin{equation*}
\begin{split}
&\int_{T_0}^\infty\int_0^{\phi_m(t)+M}\phi_m(t)^{\frac{m}{m+2}-n-\bar{\sigma}}
\big(1+\big|r-\phi_m(t)\big|\big)^{-1-\bar{\sigma}}r^{n-1}drdt \\
&\leq C\int_{T_0}^\infty\phi_m(t)^{\frac{m}{m+2}-n-\bar{\sigma}}\big(1+\phi_m(t)\big)^{n-1}dt
\leq C.
\end{split}
\end{equation*}
This, together with \eqref{equ:2.23}, yields \eqref{equ:2.2}.
\end{proof}
\subsection{Inhomogeneous Strichartz estimates without characteristic weight}
Now we turn to give inhomogeneous Strichartz estimate at \(q=q_0\) without characteristic weight, which is a key step in the proof of inhomogeneous Strichartz estimates with characteristic weight. To begin with, we introduce a result of dyadic decomposition from Lemma 3.8 of \cite{Gls1}.
\begin{lemma}\label{lem:a5}
Assume that \(\chi\in C^\infty_c(\R)\) with
\begin{equation}\label{equ:3.23-1}
\textup{supp}\chi\subseteq\left(\frac{1}{2} ,2\right),\quad  \ds \sum\limits_{j=-\infty}^\infty\chi\left(2^{-j}\tau\right) \equiv1\quad \text{for} \quad \tau>0.
\end{equation}
Define the Littlewood-Paley operators of function $G$ as follows
\[G_j(t,x)=(2\pi)^{-n}\int_{\mathbb{R}^n}e^{ix\cdot\xi}\chi\left(\frac{|\xi|}{2^j}\right)\hat{G}(t,\xi)d\xi.\]
Then one has that
\begin{align*}
\parallel G\parallel_{L^s_tL^q_x}\leq C\lc \sum\limits_{j=-\infty}^{\infty}\parallel G_j\parallel^2_{L^s_tL^q_x}\rc^{\frac{1}{2}}
\qquad \text{for\quad $2\leq q<\infty$ and $2\leq s \leq \infty$} \\
\end{align*}
and
\begin{align*}
\lc\sum\limits_{j=-\infty}^{\infty}\parallel G_j\parallel^2_{L^r_tL^p_x}\rc^{\frac{1}{2}}\leq C\parallel G\parallel_{L^r_tL^p_x}
\qquad \text{for\quad $1<p\leq2$\quad and \quad $1\leq r \leq 2$}.
\end{align*}
\end{lemma}

In next step, we turn to handle the linear inhomogeneous problem:
\begin{equation}
\begin{cases}
&\partial_t^2 w-t^m\triangle w=F(t,x), \\
&w(T_0,x)=0,\quad \partial_tw(T_0,x)=0,
\end{cases}
\label{Y-3-1}
\end{equation}
Based on Lemma 2.2, we have the following Strichartz estimate without characteristic weight:
\begin{lemma}
	For \(q_0=\frac{2((m+2)n +2+2\alpha)}{(m+2)n-2}\)
with \(n\geq2\), \(m>0\) and \(0\leq\alpha\leq\frac{n}{n-1}\cdot m\),
	\begin{equation}
\begin{split}
\lp t^{\frac{\alpha}{q_0}}w\rp_{L^{q_0}(\ra)}
\leq C\lp t^{-\frac{\alpha}{q_0}}F\rp_{L^{\frac{q_0}{q_0-1}}(\ra)}.
\end{split}
\label{equ:H.0}
\end{equation}
\end{lemma}
\begin{remark}
	Recall that in our former work \cite[Lemma 2.2]{HLYin},  we have established for \eqref{Y-3-1} with
	 \(q\geq\tilde{q}_0=\frac{2((m+2)n+2)}{(m+2)m-2-4\beta}\),
	\begin{equation}\label{equ:3.34-1}
\lp t^\beta w\rp_{L^q([1,\infty)\times\R^{n})}\leq C\,\lp t^{-\beta}|D_x|^{\gamma-\frac{1}{m+2}}
F\rp_{L^{\frac{\tilde{q}_0}{\tilde{q}_0-1}}([1,\infty)\times\R^{n})},
\end{equation}
where $\gamma=\frac{n}{2}-\frac{2\beta}{m+2} -\frac{(m+2)n+2}{q(m+2)}$, \(0<\beta\leq\frac{m}{4} \), and the constant $C=C(n,m,q)$. A direct computation shows that for the choice
 \(\beta=\frac{\alpha}{\tilde{q}_0}\),
\[\tilde{q}_0=\frac{2((m+2)n+2)}{(m+2)m-2-4\beta}=\frac{2((m+2)n+2)}{(m+2)m-2-4\frac{\alpha}{\tilde{q}_0}}\Longrightarrow  \tilde{q}_0=\frac{2((m+2)n +2+2\alpha)}{(m+2)n-2}=q_0.\]
However, the restriction \(\frac{\alpha}{q_0}=\beta\leq\frac{m}{4}\) implies
\(\alpha\leq\frac{m}{2}\cdot\frac{(m+2)n+2}{(m+2)(n-1)}\),
thus one can derive \eqref{equ:H.0} from \cite[Lemma 2.2]{HLYin} only for the case \(\alpha\leq\frac{m}{2}\cdot\frac{(m+2)n+2}{(m+2)(n-1)}\). On the other hand, it can be computed easily that for all \(n\geq3\),
\[\frac{m}{2}\cdot\frac{(m+2)n+2}{(m+2)(n-1)}<m<\frac{n}{n-1}\cdot m, \]
thus Lemma 2.3 has improved the result in \cite[Lemma 2.2]{HLYin} for the case \(q=q_0\). Furthermore, the proof of Theorem 1.1 (Case I) require the global existence result for \eqref{YH-4} with \(\alpha=m\), therefore Lemma 2.3 is a crucial step to establish \eqref{equ:3.3}.
\end{remark}

\begin{proof}[\textbf{Proof of Lemma 2.3}]

Denote \(\beta=\frac{\alpha}{q_0}\), then for \(0<\alpha\leq\frac{m}{2}\cdot\frac{(m+2)n+2}{(m+2)(n-1)}\) or equivalently \(0<\beta\leq\frac{m}{4}\), \eqref{equ:H.0} is an immediate sequence of Lemma 2.2 in our former work \cite{HLYin} . Thus it suffices to consider \(\frac{m}{2}\cdot\frac{(m+2)n+2}{(m+2)(n-1)}<\alpha\leq\frac{n}{n-1}\cdot m\) or equivalently \(\frac{m}{4}<\beta\leq\frac{m}{2}\cdot\frac{n}{n+1}\).

It follows from \cite[(2.32)]{HLYin} that
\begin{equation}\label{equ:3.23}
  w(t,x)=\int_{\ral}\int_{\R^n}
  e^{i\{(x-y)\cdot\xi\pm[\phi_m(t)-\phi_m(s)]|\xi|\}}\,
 a(t,s,\xi)F(s,y)\,d\xi
  ds dy=:AF.
\end{equation}
where
\begin{equation}\label{equ:3.24}
|\partial_\xi^\kappa a(t,s,\xi)|\lesssim  (1+\phi_m(t)|\xi|)^{-\frac{m}{2(m+2)}}
 (1+\phi_m(s)|\xi|)^{-\frac{m}{2(m+2)}}|\xi|^{-\frac{2}{m+2}-|\kappa|}.
\end{equation}
In the remaining part of this article, it is enough to consider the phase function with sign minus before
\([\phi_m(t)-\phi_m(s)]|\xi|\) since for the case of sign plus, the related estimates can be obtained in the same way.
By setting $a_\lambda(t,s,\xi)=\chi(|\xi|/\lambda)a(t,s,\xi)$ for $\lambda>0$ and \(\chi\) defined in \eqref{equ:3.23-1}, one can obtain a dyadic decomposition of the operator $A$ as follows
\begin{equation}\label{equ:3.25}
A_\lambda F=\int_{\ral}\int_{\R^n}
  e^{i\{(x-y)\cdot\xi-[\phi_m(t)-\phi_m(s)]|\xi|\}}a_\lambda(t,s,\xi)
  F(s,y)\,d\xi ds dy.
\end{equation}
Define \(b(t,s,\xi)=|\xi|a_\lambda (t,s,\xi)\). Then
\(\bigl|\partial_\xi^\kappa b (t,s,\xi)\bigr|\lesssim t^{-\frac{m}{4}}s^{-\frac{m}{4}}\lambda ^{-|\kappa|}\)
	and
\begin{equation}\label{equ:3.26}
A_\lambda F=\int_{\ral}\int_{\R^n}
  e^{i\{(x-y)\cdot\xi-[\phi_m(t)-\phi_m(s)]|\xi|\}}
  |\xi|^{-1}b(t,s,\xi)F(s,y)\,d\xi dyds.
\end{equation}
Set
\begin{equation}\label{equ:3.27}
	H_{t,s}f(x)=\int_{\R^n}\int_{\R^n} e^{i\{(x-y)\cdot\xi-[\phi_m(t)-\phi_m(s)]|\xi|\}}b(t,s,\xi)f(y)\,d\xi dy,
\end{equation}
then \(|A_\lambda F|\lesssim|\lambda^{-1}\int_{T_0}^tH_{t,s}F\md s|\). Since \(t\geq s\geq T_0\), \eqref{equ:3.24} yields \(|\partial_\xi^\kappa b(t,s,\xi)|\lesssim t^{-\frac{m}{4}}s^{-\frac{m}{4}}|\xi|^{-|\kappa|}\), hence
\begin{equation}\label{equ:3.35}
\lp t^{\beta}H_{t,s}f(\cdot)\rp_{L^2(\R^n)}\leq C t^{\beta-\frac{m}{4}}s^{\beta-\frac{m}{4}}\left\| s^{-\beta}f(\cdot)\right\|_{L^2(\R^n)}.
\end{equation}
In addition, by the method of stationary phase, one has
\begin{equation}\label{equ:3.36}
\begin{aligned}
\Big\| t^{\beta} H_{t,s}&f(\cdot)\Big\|_{L^\infty(\R^n)}\lesssim \lambda^n
\lc1+\lambda \left|\phi_m(t)-\phi_m(s)\right|\rc^{-\frac{n-1}{2}}t^{\beta-\frac{m}{4}}s^{\beta-\frac{m}{4}}\left\| s^{-\beta}f(\cdot)\right\|_{L^1(\R^n)} \\
&\lesssim\lambda^n
(1+\lambda \left|\phi_m(t)-\phi_m(s)\right|)^{-\left(n-\frac{(m+2)n+2}{4(\beta+1)}\right)}t^{\beta-\frac{m}{4}}s^{\beta-\frac{m}{4}}\left\| s^{-\beta}f(\cdot)\right\|_{L^1(\R^n)} \\
&\lesssim\lambda^{\frac{(m+2)n+2}{4(\beta+1)}}
\left|\phi_m(t)-\phi_m(s)\right|^{ -\left(n-\frac{(m+2)n+2}{4(\beta+1)}\right)}
t^{\beta-\frac{m}{4}}s^{\beta-\frac{m}{4}}\left\| s^{-\beta}f(\cdot)\right\|_{L^1(\R^n)},
\end{aligned}
\end{equation}
where in the first inequality we use the fact \(0<n-\frac{(m+2)n+2}{4(\beta+1)}\leq\frac{n-1}{2}\) provided \(\frac{m}{4}< \beta\leq\frac{m}{2}\cdot\frac{n}{n+1}\).
To proceed further, we shall apply different techniques according to the value of \(\frac{s}{t}\).

\paragraph{Case I \(\mathbf { 2s\leq t }\)}

Since \(t-s\geq\frac{t}{2}\) and \(\beta>\frac{m}{4}\) , we have \(t^{\beta-\frac{m}{4}}s^{\beta-\frac{m}{4}}\lesssim|t-s|^{2\beta-\frac{m}{2}} \). Thus \eqref{equ:3.35} implies
\begin{equation}\label{equ:3.35-1}
\lp t^{\beta}H_{t,s}f(\cdot)\rp_{L^2(\R^n)}\leq C |t-s|^{2\beta-\frac{m}{2}}\left\| s^{-\beta}f(\cdot)\right\|_{L^2(\R^n)},
\end{equation}
and \eqref{equ:3.36} yields
\begin{equation}\label{equ:3.36-1}
	\lp t^{\beta}H_{t,s}f(\cdot)\rp_{L^\infty(\R^n)}\leq C\lambda^{\frac{(m+2)n+2}{4(\beta+1)}}
\left|t-s\right|^{ -\left(n-\frac{(m+2)n+2}{4(\beta+1)}\right)\frac{m+2}{2}}
|t-s|^{2\beta-\frac{m}{2}}\left\| s^{-\beta}f(\cdot)\right\|_{L^1(\R^n)}.
\end{equation}
Interpolating \eqref{equ:3.35-1} with \eqref{equ:3.36-1}, we have
\begin{equation}\label{equ:3.32}
  \lp t^{\beta}H_{t,s}f(\cdot)\rp_{L^{q_0}(\R^n)}\leq C\lambda
  \left|t-s\right|^{-\frac{(m+2)n-2-4\beta}{(m+2)n+2}}\left\| s^{-\beta}f(\cdot)\right\|_{L^{p_0}(\R^n)}.
\end{equation}

\paragraph{Case II \(\mathbf { 2s>t }\)}

By mean value theorem we have
\[\phi_m(t)-\phi_m(s)=\eta^{\frac{m}{2}}(t-s), \quad \eta\in[s,t]. \]
Note that \(\eta\geq s>\frac{t}{2}\), then by \eqref{equ:3.36} we get
\begin{equation}\label{equ:3.36-2}
\begin{aligned}
&\lp t^{\beta} H_{t,s}f(\cdot)\rp_{L^\infty(\R^n)} \\
&\leq C\lambda^{\frac{(m+2)n+2}{4(\beta+1)}}\eta^{ -\frac{m}{2}\left(n-\frac{(m+2)n+2}{4(\beta+1)}\right)}
\left|t-s\right|^{ -\left(n-\frac{(m+2)n+2}{4(\beta+1)}\right)}
t^{\beta-\frac{m}{4}}s^{\beta-\frac{m}{4}}\left\| s^{-\beta}f(\cdot)\right\|_{L^1(\R^n)} \\
&\leq C\lambda^{\frac{(m+2)n+2}{4(\beta+1)}}t^{ -\frac{m}{2}\left(n-\frac{(m+2)n+2}{4(\beta+1)}\right)}
\left|t-s\right|^{ -\left(n-\frac{(m+2)n+2}{4(\beta+1)}\right)}
t^{2\beta-\frac{m}{2}}\left\| s^{-\beta}f(\cdot)\right\|_{L^1(\R^n)},
\end{aligned}
\end{equation}
also \eqref{equ:3.35} gives
\begin{equation}\label{equ:3.35-2}
\lp t^{\beta}H_{t,s}f(\cdot)\rp_{L^2(\R^n)}\leq C t^{2\beta-\frac{m}{2}}\left\| s^{-\beta}f(\cdot)\right\|_{L^2(\R^n)}.
\end{equation}
Interpolating \eqref{equ:3.35-2} with \eqref{equ:3.36-2}, we have
\begin{equation}\label{equ:3.32-1}
\begin{split}
	&\lp t^{\beta}H_{t,s}f(\cdot)\rp_{L^{q_0}(\R^n)} \\
	&\leq C\lambda t^{2\beta-\frac{m}{2}
  -\frac{m}{2}\left(n-\frac{(m+2)n+2}{4(\beta+1)}\right)\frac{4(\beta+1)}{(m+2)n+2}}
  \left|t-s\right|^{-\left(n-\frac{(m+2)n+2}{4(\beta+1)}\right)\frac{4(\beta+1)}{(m+2)n+2}}\left\| s^{-\beta}f(\cdot)\right\|_{L^{p_0}(\R^n)}.
\end{split}
  \end{equation}
By \(\beta\leq\frac{m}{2}\cdot\frac{n}{n+1}\) we compute
\[2\beta-\frac{m}{2}
  -\frac{m}{2}\left(n-\frac{(m+2)n+2}{4(\beta+1)}\right)\frac{4(\beta+1)}{(m+2)n+2}
  =2\beta-\frac{2mn(\beta+1)}{(m+2)n+2}\leq0,\]
thus \eqref{equ:3.32-1} gives
\begin{equation}\label{equ:3.32-2}
\begin{split}
	&\lp t^{\beta}H_{t,s}f(\cdot)\rp_{L^{q_0}(\R^n)} \\
	&\leq C\lambda \left|t-s\right|^{2\beta-\frac{2mn(\beta+1)}{(m+2)n+2}}
  \left|t-s\right|^{-\left(n-\frac{(m+2)n+2}{4(\beta+1)}\right)\frac{4(\beta+1)}{(m+2)n+2}}\left\| s^{-\beta}f(\cdot)\right\|_{L^{p_0}(\R^n)} \\
  &\leq C\lambda
  \left|t-s\right|^{-\frac{(m+2)n-2-4\beta}{(m+2)n+2}}\left\| s^{-\beta}f(\cdot)\right\|_{L^{p_0}(\R^n)}.
\end{split}
  \end{equation}

Collecting \eqref{equ:3.32} and \eqref{equ:3.32-2}, then by $1-(\frac{1}{p_0}-\frac{1}{q_0})=
\frac{(m+2)n-2-4\beta}{(m+2)n+2}$ and the
Hardy-Littlewood-Sobolev inequality, we arrive at
\begin{multline}\label{equ:3.33-1}
\| A_\lambda F\|_{L^{q_0}(\ra)}
=\left\| \lambda^{-1}\int_{T_0}^{\infty}H_{t,s}F\,ds \right\|_{L^{q_0}(\ra)}\\
\begin{aligned}
&\leq C\lambda^{-1}\lambda\left\|\int_{T_0}^{\infty}
  |t-s|^{-\frac{(m+2)n-2-4\beta}{(m+2)n+2}}\left\| F(s,\cdot)\right\|_{L^{p_0}(\R^n)}\,
  ds\right\|_{L^{q_0}([T_0,\infty))} \\
&\leq C\left\| F\right\|_{L^{p_0}(\ra)}.
\end{aligned}
\end{multline}
It follows from Lemma 2.2 and
$p_0=\frac{2((m+2)n+2)}{(m+2)n+6+4\beta}<2$ that
\begin{equation}\label{equ:combine}
\begin{aligned}
\| A F \|_{L^{q_0}}^2&\leq C\sum_{j\in\Z}\| A_{2^j}F\|_{L^{q_0}}^2
\leq C\sum_{j\in\Z}\sum_{k:|j-k|\leq C_0}\| A_{2^j}F_k\|_{L^{q_0}}^2 \\
&\leq C\sum_{j\in\Z}\sum_{k:|j-k|\leq C_0}\| F_k\|_{L^{p_0}}^2
\le C\, \| F\|_{L^{p_0}(\ra)}^2,
\end{aligned}
\end{equation}
where $\hat{F}_k(s,\xi)=\chi(2^{-k}|\xi|)\,\hat{F}(s,\xi)$.
Hence,  the proof of the Lemma is completed.	
\end{proof}
As the end of this section, we prove \eqref{equ:H.0} with \(-1<\alpha<0\), for technical reason, we only show the local in time case.
\begin{lemma}
	Let \(n\geq2\) and \(-1<\alpha<0\), then for any fixed large \(\bar{T}>0\),
	\begin{equation}\label{equ:4.102}
\begin{split}
& \lp t^{\frac{\alpha}{q_0}}w\rp_{L^{q_0}([T_0, \bar{T}]\times\mathbb{R}^n)}\leq C \lp t^{-\frac{\alpha}{q_0}}F\rp_{L^{\frac{q_0}{q_0-1}}([T_0, \bar{T}]\times\mathbb{R}^n)},
\end{split}
\end{equation}
where \(C\) depends on \(q_0\) and \(\bar{T}\).
\end{lemma}
\begin{proof}
Note that \(w\) can be written in \eqref{equ:3.23} with the amplitude function \(a(t,s,\xi)\) satisfying \eqref{equ:3.24},

Define the dyadic operator \(A_\lambda\) as in \eqref{equ:3.25}, and set
\begin{equation}\label{equ:3.28}
	\tilde{H}_{t,s}f(x)=\int_{\R^n}\int_{\R^n} e^{i\{(x-y)\cdot\xi-[\phi_m(t)-\phi_m(s)]|\xi|\}}a_\lambda (t,s,\xi)f(y)\,d\xi dy.
\end{equation}
Then \(A_\lambda G=\int_{T_0}^t\tilde{H}_{t,s}G\md s\) and \eqref{equ:3.24} yields
\begin{equation}\label{equ:3.35-3}
\lp t^{\beta}\tilde{H}_{t,s}f(\cdot)\rp_{L^2(\R^n)}\leq Ct^\beta s^\beta \lambda^{-\frac{2}{m+2}} \left\| s^{-\beta}f(\cdot)\right\|_{L^2(\R^n)}\leq C \lambda^{-\frac{2}{m+2}} \left\| s^{-\beta}f(\cdot)\right\|_{L^2(\R^n)}.
\end{equation}
In addition, for \(\lambda<1\) , one has
\begin{equation}\label{equ:3.36-3}
\begin{aligned}
\lp t^{\beta} \tilde{H}_{t,s}f(\cdot)\rp_{L^\infty(\R^n)}&\leq C\lambda^{n-\frac{2}{m+2}}
t^\beta s^\beta \left\| s^{-\beta}f(\cdot)\right\|_{L^1(\R^n)}
\leq C\lambda^{n-\frac{2}{m+2}}\left\| s^{-\beta}f(\cdot)\right\|_{L^1(\R^n)}.
\end{aligned}
\end{equation}
Interpolation \eqref{equ:3.35-3} with \eqref{equ:3.36-3}, we get
\begin{equation}
	\lp t^{\beta} \tilde{H}_{t,s}f(\cdot)\rp_{L^{q_0}(\R^n)}\leq C\lambda^{\frac{4n(\beta+1)}{(m+2)n+2}-\frac{2}{m+2}}\left\| s^{-\beta}f(\cdot)\right\|_{L^{p_0}(\R^n)}.
\end{equation}
The condition \(\alpha>-1\) implies that
\[\frac{4n(\beta+1)}{(m+2)n+2}-\frac{2}{m+2}>\frac{4n}{(m+2)n+2}\cdot\frac{(m+2)n+2}{2(m+2)n}-\frac{2}{m+2}=0, \]
then by \(1\leq s\leq t\leq\bar{T}\) we have for \(\lambda<1\)
\begin{equation}
	\lp t^{\beta} \tilde{H}_{t,s}f(\cdot)\rp_{L^{q_0}(\R^n)}\leq C(\bar{T})\left|t-s\right|^{-\frac{(m+2)n-2-4\beta}{(m+2)n+2}}\left\| s^{-\beta}f(\cdot)\right\|_{L^{p_0}(\R^n)}.
\end{equation}
On the other hand, if \(\lambda\geq1\), then by stationary phase method, we get
\begin{equation}\label{equ:3.36-4}
\begin{aligned}
\lp t^{\beta} \tilde{H}_{t,s}f(\cdot)\rp_{L^\infty(\R^n)}&\leq C\lambda^n
(1+\lambda \left|\phi_m(t)-\phi_m(s)\right|)^{-\frac{n-1}{2}}t^{\beta-\frac{m}{4}}s^{\beta-\frac{m}{4}}\lambda^{-1} \left\| s^{-\beta}f(\cdot)\right\|_{L^1(\R^n)} \\
&\leq C\lambda^{\frac{n+1}{2}}\left|t-s\right|^{-\frac{n-1}{2}}
\lambda^{-1} \left\| s^{-\beta}f(\cdot)\right\|_{L^1(\R^n)},
\end{aligned}
\end{equation}
while for the \(L^2\) estimate, we have
\begin{equation}\label{equ:3.35-4}
	\lp t^{\beta}\tilde{H}_{t,s}f(\cdot)\rp_{L^2(\R^n)}\leq C t^{\beta-\frac{m}{4}}s^{\beta-\frac{m}{4}}\lambda^{-1} \left\| s^{-\beta}f(\cdot)\right\|_{L^2(\R^n)}\leq C\lambda^{-1} \left\| s^{-\beta}f(\cdot)\right\|_{L^2(\R^n)}.
\end{equation}
Interpolation \eqref{equ:3.35-4} with \eqref{equ:3.36-4}, we arrive at
\begin{equation}
	\lp t^{\beta} \tilde{H}_{t,s}f(\cdot)\rp_{L^{q_0}(\R^n)}\leq C\lambda^{\frac{2(n+1)(\beta+1)}{(m+2)n+2}-1}\left|t-s\right|^{-\frac{2(n-1)(\beta+1)}{(m+2)n+2}}
	\left\| s^{-\beta}f(\cdot)\right\|_{L^{p_0}(\R^n)}.
\end{equation}
Since \(\beta<0\), one can compute \(\frac{2(n+1)(\beta+1)}{(m+2)n+2}-1<0\) and \(-\frac{2(n-1)(\beta+1)}{(m+2)n+2}>-\frac{(m+2)n-2-4\beta}{(m+2)n+2}\), then by \(|t-s|\leq2\bar{T}\) we have
\begin{equation}
	\lp t^{\beta} \tilde{H}_{t,s}f(\cdot)\rp_{L^{q_0}(\R^n)}\leq C(\bar{T})\left|t-s\right|^{-\frac{(m+2)n-2-4\beta}{(m+2)n+2}}\left\| s^{-\beta}f(\cdot)\right\|_{L^{p_0}(\R^n)}
\end{equation}
for \(\lambda\geq1\). The remaining part of the proof is similar to \eqref{equ:3.33-1}-\eqref{equ:combine} in Lemma 2.3, we omit the details.
\end{proof}

\section{The proof of Theorem 1.3 at the end point \(\mathbf{q=q_0}\)}\label{sec4}
In order to prove Theorem~\ref{thm:inhomogeneous estimate}, as stated in Section 1.4,
it suffices to handle the two endpoint cases, which correspond to \(q=q_0=\frac{2((m+2)n+2\alpha +2)}{(m+2)n-2}\) and \(q=2\).
We start with the proof of \eqref{equ:3.3}.

\subsection{Local estimate at \(q=q_0\)}

One can write inequality \eqref{equ:3.3} as
\begin{equation}\label{equ:4.1}
\left\|\lc\phi_m^2(t)-|x|^2\rc^{\frac{1}{q_0}-\nu}t^{\frac{\alpha}{q_0}}w\right\|_{L^{q_0}([T_0,\infty)\times\mathbb{R}^n)}
\leq C\left\|\lc\phi_m^2(t)-|x|^2\rc^{\frac{1}{q_0}+\nu}t^{-\frac{\alpha}{q_0}}F\right\|_{L^{\frac{q_0}{q_0-1}}(\ra)},
\end{equation}
where $\nu>0$. First note that for any fixed \(\bar{T}\gg T_0\), the weight \(\phi_m^2(t)-|x|^2\) and \(t^{\frac{\alpha}{q_0}}\) are both bounded from below and above for \(T_0\leq t\leq \bar{T}\), this observation together with Lemma 2.3 and Lemma 2.4 give
\begin{equation}\label{equ:4.102}
\begin{split}
& \lp\lc\phi_m^2(t)-|x|^2\rc^{\frac{1}{q_0}-\nu}t^{\frac{\alpha}{q_0}}w\rp_{L^{q_0}([T_0, \bar{T}]\times\mathbb{R}^n)} \\
& \leq C\phi_m(\bar{T})^{-\frac{\nu}{4}} \lp\lc\phi_m^2(t)-|x|^2\rc^{\frac{1}{q_0}+\nu}t^{-\frac{\alpha}{q_0}}F\rp_{L^{\frac{q_0}{q_0-1}}([T_0, \bar{T}]\times\mathbb{R}^n)}.
\end{split}
\end{equation}
Hence in order to prove \eqref{equ:4.1} and further \eqref{equ:3.3}, it suffices to prove the following inequality for $T\geq \bar{T}$,
\begin{equation}
\begin{split}
&\Big\|\big(\phi_m^2(t)-|x|^2\big)^{\frac{1}{q_0}-\nu}t^{\frac{\alpha}{q_0}}w\Big\|_{L^{q_0}([T,2T]\times\mathbb{R}^n)} \\
& \leq C\phi_m(T)^{-\frac{\nu}{4}} \Big\|\big(\phi_m^2(t)-|x|^2\big)^{\frac{1}{q_0}+\nu}
t^{-\frac{\alpha}{q_0}}F\Big\|_{L^{\frac{q_0}{q_0-1}}
([T_0,\infty)\times\mathbb{R}^n)},
\end{split}
\label{equ:4.2}
\end{equation}
where $\bar{T}>0$ is a fixed large constant.
\subsection{Simplifications for the end point estimate}
\label{sect: Y}

Note that $F(t,x)\equiv0$ for $|x|>\phi_m(t)-1$, then this means
$\operatorname{supp}\ F\subseteq\{(t,x): |x|^2\leq\phi_m^2(t)-1\}$.
Set $F=F^0+F^1$, where
\begin{equation}
F^0=F, \quad \text{if}\quad t\geq\frac{T}{2\cdot10^{\frac{2}{m+2}}}; \quad F^0=0, \quad \text{if}\quad t<\frac{T}{2\cdot10^{\frac{2}{m+2}}}.
\label{equ:4.30}
\end{equation}
Correspondingly, let $w=w^0+w^1$, where $\partial_t^2w^j-t^m\Delta w^j=F^j$ with zero data $(j=0,1)$. Hence in order to prove \eqref{equ:4.2}, it suffices to show that for \(j=0,1\),
\begin{equation}
\begin{split}
\Big\|&\big(\phi^2_m(t)-|x|^2\big)^{\frac{1}{q_0}-\nu}t^{\frac{\alpha}{q_0}}w^j\Big\|_{L^{q_0}(\{(t,x):T\leq t\leq 2T\})} \\
&\leq C\phi_m(T)^{-\frac{\nu}{4}}\Big\|\big(\phi^2_m(t)-|x|^2\big)^{\frac{1}{q_0}+\nu}t^{-\frac{\alpha}{q_0}}F^j\Big\|_{L^{\frac{q_0}{q_0-1}}(\ra)}.\\
\end{split}
\label{equ:4.34}
\end{equation}

For this purpose, we shall make some reductions by which we restrict the support of \(F_j\) and \(w_j\)
in certain domains, such that in each domain the characteristic weight $\phi^2_m(t)-|x|^2$
on both sides of \eqref{equ:4.34} are essentially constants and hence can be removed. More specifically, following the idea of \cite{Gls1} and \cite{HWY4}, we assume that $\operatorname{supp}F^j\subseteq [\bar{T_0}, 2\bar{T_0}]\times\mathbb{R}^n$ for $\bar{T_0}=2^k\bar{T}$, \(k=0,1,2,...\) such that \(\bar{T_0}\leq T\), and $F^j\equiv0$ holds
when $\phi_m(t)-|x|\notin[\delta_0\phi_m(\bar T_0), 2\delta_0\phi_m(\bar{T_0})]$ for some fixed
constant $\dl_0$ with $0<\delta_0\le 2$ and $\delta_0\phi_m(\bar{T_0})\geq1$; while for the solution \(w_j\), we further assume that $w^j\equiv0$ holds
when $\phi_m(t)-|x|\notin[\delta\phi_m(\bar T_0), 2\delta\phi_m(\bar{T_0})]$ for $\delta\geq\delta_0$ such that \(\delta\phi_m(\bar{T}_0)\leq\phi_m(T) \).
 With these reductions, our task is reduced to prove some unweighted Strichartz estimates

\begin{equation}
\begin{split}
\Big\|&\big(\phi^2_m(t)-|x|^2\big)^{\frac{1}{q_0}-\nu}t^{\frac{\alpha}{q_0}}w^j\Big\|_{L^{q_0}(\{(t,x):T\leq t\leq 2T, \delta\phi_m(\bar{T_0})\leq\phi_m(t)-|x|\leq
 2\delta\phi_m(\bar{T_0})\})} \\
&\leq C\lc\phi_m(\bar{T_0})\phi_m(T)\rc^{-\frac{\nu}{2}}
\Big\|\big(\phi^2_m(t)-|x|^2\big)^{\frac{1}{q_0}+\nu}t^{-\frac{\alpha}{q_0}}F^j\Big\|_{L^{\frac{q_0}{q_0-1}}(\ra)}.
\end{split}
\label{equ:4.37}
\end{equation}
 By the scale of \(\phi_m(t)\) and  \(\phi_m(t)-|x|\), \eqref{equ:4.37} is equivalent to
\begin{align}\label{equ:4.A}
\big(\phi_m(\bar{T_0})&\phi_m(T)\delta\big)^{\frac{1}{q_0}-\nu}\lp  t^{\frac{\alpha}{q_0}}
w^j\rp_{L^{q_0}(\{(t,x):T\leq t\leq 2T, \delta\phi_m(\bar{T_0})\leq\phi_m(t)-|x|\leq
 2\delta\phi_m(\bar{T_0})\})} \no\\
\leq C&\lc\phi_m(\bar{T_0})\phi_m(T)\rc^{-\frac{\nu}{2}}\lc\phi^2_m(\bar{T_0})\delta_0\rc^{\frac{1}{q_0}+\nu}
\lp t^{-\frac{\alpha}{q_0}}F^j\rp_{L^{\frac{q_0}{q_0-1}}(\ra)}.
\end{align}
By rearranging some terms in \eqref{equ:4.A}, then \eqref{equ:4.A} directly follows from
\begin{align}\label{equ:4.B}
\Big(\frac{\phi_m(T)}{\phi_m(\bar{T_0})}\Big)^{\frac{1}{q_0}-\frac{\nu}{2}}\delta^{\frac{1}{q_0}+\frac{\nu}{2}}\frac{1}{\phi_m(\bar{T_0})^{3\nu}\delta^{\f{3}{2}\nu}\delta_0^{\nu}}
&\lp t^{\frac{\alpha}{q_0}}w^j\rp_{L^{q_0}(\{(t,x):T\leq t\leq 2T, \delta\phi_m(\bar{T_0})\leq\phi_m(t)-|x|\leq2\delta\phi_m(\bar{T_0})\})} \no\\
&\leq C\delta_0^{\frac{1}{q_0}}\lp t^{-\frac{\alpha}{q_0}}F^j\rp_{L^{\frac{q_0}{q_0-1}}(\ra)}.
\end{align}
Note that
\(\phi_m(\bar{T_0})^{3\nu}\delta^{\f{3}{2}\nu}\delta_0^{\nu}\geq\phi_m(\bar{T_0})^{-3\nu}\delta_0^{-\f{5}{2}\nu}
\geq\delta_0^{\f{\nu}{2}}\gtrsim 1 .\)
Therefore, \eqref{equ:4.B} follows from
\begin{equation}\label{equ:4.38}
\begin{split}
\Big(\frac{\phi_m(T)}{\phi_m(\bar{T_0})}\Big)^{\frac{1}{q_0}-\frac{\nu}{2}}
\delta^{\frac{1}{q_0}+\frac{\nu}{2}}
\|& t^{\frac{\alpha}{q_0}}w^j\|_{L^{q_0}(\{(t,x):T\leq t\leq 2T, \delta\phi_m(\bar{T_0})\leq\phi_m(t)-|x|\leq2\delta\phi_m(\bar{T_0})\})} \\
&\leq C\delta_0^{\frac{1}{q_0}}\| t^{-\frac{\alpha}{q_0}}F^j\|_{L^{\frac{q_0}{q_0-1}}(\ra)}.
\end{split}
\end{equation}

To proceed further, we set $G^j(t,x)=:\bar{T_0}^2F^j(\bar{T_0}t, \bar{T_0}^{\frac{m+2}{2}}x)$ and
$v^j(t,x)=:w^j(\bar{T_0}t, \bar{T_0}^{\frac{m+2}{2}}x)$ for \(j=0,1\).
Then $v^j$ satisfies
\begin{equation}\label{equ:4.104}
\begin{cases}
&\partial_t^2 v^j-t^m\triangle v^j=G^j(t,x), \\
&v^j(0,x)=0,\quad \partial_tv^j(0,x)=0, \\
\end{cases}
\end{equation}
where $\operatorname{supp}G^j\subseteq\{(t,x):1\leq t\leq2,\delta_0\phi_m(1)\leq \phi_m(t)-|x|\leq2\delta_0\phi_m(1)\}$. Then, if we let $T/\bar{T_0}$ denoted
by $T$, then \eqref{equ:4.38} is a result of
\begin{equation}
\begin{split}
	\phi_m(T)^{\frac{1}{q_0}-\frac{\nu}{2}}\delta^{\frac{1}{q_0}+\frac{\nu}{2}}&
\lp t^{\frac{\alpha}{q_0}}v^j\rp_{L^{q_0}(\{(t,x):T\leq t\leq 2T, \delta\phi_m(1)\leq\phi_m(t)-|x|\leq2\delta\phi_m(1)\})} \\
& \leq C\delta_0^{\frac{1}{q_0}}\lp t^{-\frac{\alpha}{q_0}}G^j\rp_{L^{\frac{q_0}{q_0-1}}(\rb)}.
\end{split}
\label{equ:4.85}
\end{equation}

At this time, by \eqref{equ:4.30}, \(1\leq T/\bar{T_0}\leq4\cdot10^{\frac{2}{m+2}} \) holds  for $(t,x)\in \operatorname{supp}w^0$ and \(T\leq t\leq 2T\),
or equivalently, \(1\leq T\leq 4\cdot10^{\frac{2}{m+2}}\) holds  for $(t,x)\in \operatorname{supp}v^0$ and
\(T\leq t\leq 2T\), which is called
the relatively ``small times".

On the other hand,
we have that \(T/\bar{T_0}\geq 2\cdot10^{\frac{2}{m+2}}\)
holds for $(t,x)\in \operatorname{supp}w^1$ and \(T\leq t\leq 2T\), or equivalently, \(T\geq 2\cdot10^{\frac{2}{m+2}}\) holds for $(t,x)\in \operatorname{supp}v^1$ and \(T\leq t\leq 2T\), which is called the relatively ``large times".

In Subsection 4.2 and Subsection 4.3, we will handle the two cases
respectively. For the concision of notation, in the following subsections, we will omit the superscript \(j\) and denote
\begin{equation}\label{equ:D-tx}
	D_{t,x}^{T,\delta}=\{(t,x):T\leq t\leq 2T, \delta\phi_m(1)\leq\phi_m(t)-|x|\leq2\delta\phi_m(1)\}.
\end{equation}
Then our task is reduced to prove
\begin{equation}
\phi_m(T)^{\frac{1}{q_0}-\frac{\nu}{2}}\delta^{\frac{1}{q_0}+\frac{\nu}{2}}
\lp t^{\frac{\alpha}{q_0}}v\rp_{L^{q_0}(D_{t,x}^{T,\delta})}
\leq C\delta_0^{\frac{1}{q_0}}\lp t^{-\frac{\alpha}{q_0}}G\rp_{L^{\frac{q_0}{q_0-1}}(\rb)}.
\label{equ:4.39}
\end{equation}

\subsection{Some related estimates for small times}
Let us begin with the estimate of \(v^0\) in \eqref{equ:4.39}.
Note that the integral domain of $v$ satisfies $\delta\phi_m(1)\leq\phi_m(t)-|x|\leq2\delta\phi_m(1)$
and the support of $G$ satisfies $\delta_0\phi_m(1)\leq\phi_m(t)-|x|\leq2\delta_0\phi_m(1)$, in order to
treat the related Fourier integral operator (corresponding to the estimate of $w^0$),
we distinguish the following two cases according to
the different values of $\delta/\delta_0$:

(i) $\delta_0\leq\delta\leq10\cdot 2^{\frac{m+2}{2}}\delta_0$;

(ii) $10\cdot 2^{\frac{m+2}{2}}\delta_0\leq\delta\leq (2T)^{\frac{m+2}{2}}$, if \(T^{\frac{m+2}{2}}\geq10\). \\

\paragraph{Case (i): small \(\delta\)}

In this case, one has that $\delta/\delta_0\in[1,10\cdot 2^{\frac{m+2}{2}}]$ and $\delta\phi_m(\bar{T_0})\leq\phi_m(T)$ in the support of \(w\) ( \(\delta\leq T^{\frac{m+2}{2}}\) in the support of \(v\) ). Thus in order to prove \eqref{equ:4.39}, it suffices to show
\begin{equation}\label{D1}
\phi_m(T)^{\frac{1}{q_0}}\lp t^{\frac{\alpha}{q_0}}v\rp_{L^{q_0}(D_{t,x}^{T,\delta})}\leq C\lp t^{-\frac{\alpha}{q_0}}G\rp_{L^{\frac{q_0}{q_0-1}}(\rb)}.
\end{equation}
By $\phi_m(1)\leq\phi_m(T)\leq10\phi_m(4)$, we only need to prove
\begin{equation}\label{D2}
\lp t^{\frac{\alpha}{q_0}}v\rp_{L^{q_0}(D_{t,x}^{T,\delta})}\leq C\lp t^{-\frac{\alpha}{q_0}}G\rp_{L^{\frac{q_0}{q_0-1}}(\rb)},
\end{equation}
and \eqref{D2} follows from Lemma 2.3 and Lemma 2.4 immediately.

\paragraph{Case (ii): large \(\delta\)}

In this case, we only need to prove
\begin{equation}
\delta^{\frac{1}{q_0}}\lp t^{\frac{\alpha}{q_0}}v\rp_{L^{q_0}(D_{t,x}^{T,\delta})}\leq C\delta_0^{\frac{1}{q_0}}\lp t^{-\frac{\alpha}{q_0}}G\rp_{L^{\frac{q_0}{q_0-1}}(\rb)}.
\label{equ:4.68}
\end{equation}
By \eqref{equ:3.23} , we can write
\begin{equation}\label{D4}
\begin{split}
v=&\int_{\rb}\int_{\Bbb R^n}e^{i\{(x-y)\cdot\xi-[\phi_m(t)-\phi_m(s)]|\xi|\}}a(t,s,\xi)G(s,y) \md\xi\md y\md s.
\end{split}
\end{equation}
Similar to the proof of Lemma 2.3, by setting  $a_\lambda(t,s,\xi)=\chi(|\xi|/\lambda)a(t,s,\xi)$ for $\lambda>0$ with function \(\chi\) defined in \eqref{equ:3.23-1}, then one can obtain a dyadic decomposition of $v$ as follows
\begin{equation}\label{equ:3.25-1}
v_\lambda=\int_{\rb}\int_{\R^n}e^{i\{(x-y)\cdot\xi-[\phi_m(t)-\phi_m(s)]|\xi|\}}a_\lambda(t,s,\xi)
  G(s,y)\,d\xi ds dy.
\end{equation}
Then it suffices to prove \eqref{equ:4.68} for \(v_\lambda\). Apply the operator \(\tilde{H}_{t,s}\) defined in \eqref{equ:3.28} to \(G\):
\[
\tilde{H}_{t,s}G(x)=\int_{\R^n}\int_{\R^n} e^{i\{(x-y)\cdot\xi-[\phi_m(t)-\phi_m(s)]|\xi|\}}a_\lambda(t,s,\xi)G(s,y)\,d\xi dy,
\]
then similar analysis as that in the derivation of \eqref{equ:3.32} and \eqref{equ:3.32-2} in Lemma 2.3 gives
\begin{equation}\label{equ:3.32-3}
\begin{split}
	&\lp t^{\beta}\tilde{H}_{t,s}G(\cdot)\rp_{L^{q_0}(\R^n)}\leq C
  \left|t-s\right|^{-\frac{(m+2)n-2-4\beta}{(m+2)n+2}}\left\| s^{-\beta}G(s,\cdot)\right\|_{L^{p_0}(\R^n)},
\end{split}
  \end{equation}
 where \(\beta=\frac{\alpha}{q_0}\).
Then by $\frac{(m+2)n-2-4\beta}{(m+2)n+2}=\frac{(m+2)n-2}{(m+2)n+2+2\alpha}=1-(\frac{1}{p_0}-\frac{1}{q_0})$ and the
Hardy-Littlewood-Sobolev inequality, we arrive at
\begin{equation}\label{equ:3.33-2}
\begin{split}
\lp t^{\frac{\alpha}{q_0}}v_\lambda\rp_{L^{q_0}(D_{t,x}^{T,\delta})}
&\leq\left\| \int t^{\frac{\alpha}{q_0}}\tilde{H}_{t,s}G\,ds \right\|_{L^{q_0}(\rb)}\\
&\leq C\left\|\int_I
  |t-s|^{-\frac{2}{q_0}}\left\| s^{-\beta}G(s,\cdot)\right\|_{L^{p_0}(\R^n)}\,
  ds\right\|_{L^{q_0}([1,2])} \\
&\leq C\left\||s|^{-\frac{2}{q_0}}\right\| _{L^{\frac{q_0}{2}}(I)}\left\| s^{-\beta}G\right\|_{L^{p_0}(\rb)}.
\end{split}
\end{equation}
Since the support of \(G(s,y)\) satisfies \(\delta_0\phi_m(1)\leq \phi_m(s)-|y|\leq2\delta_0\phi_m(1)\) and
\[\phi_m(s')-\phi_m(s)=(s'-s)\eta^{\frac{m}{2}}, \eta\in[1,2], \]
the length of the integral interval \(I\) should be controlled by \(c\delta_0\)  with some fixed positive constant \(c\). With out loss of generality, we may assume \(c=1\). Furthermore, note that \(|\phi_m(t)-\phi_m(s)|\gtrsim \delta \) and hence \(|t-s|\geq C\delta \) for some fixed positive constant \(C\). Therefore
\begin{equation}\label{equ:2.39}
	\left\||s|^{-\frac{2}{q_0}}\right\| _{L^{\frac{q_0}{2}}(I)}\leq\left(\int_{C\delta}^{C\delta+\delta_0}\frac{1}{|s|}\md s \right)^{\frac{2}{q_0}}\lesssim\left(\frac{\delta_0}{\delta} \right)^{\frac{2}{q_0}}\leq\left(\frac{\delta_0}{\delta} \right)^{\frac{1}{q_0}}.
\end{equation}
Combing \eqref{equ:3.33-2} and \eqref{equ:2.39}, then \eqref{equ:4.68} and further \eqref{equ:4.39} is proved.

\textbf{Thus \eqref{equ:4.39} has been established for relatively small time.}

\subsection{Some related estimates for large times}
\label{sec4:large}
We now deal with the cases of ``relative large times" in \eqref{equ:4.39}, for which $\phi_m(T)\geq10\phi_m(2)$.
Note that the integral domain of $v$ satisfies $\{\delta\phi_m(1)\leq\phi_m(t)-|x|\leq2\delta\phi_m(1)\}$, while
the support of $G$ satisfies $\{\delta_0\phi_m(1)\leq\phi_m(t)-|x|\leq2\delta_0\phi_m(1)\}$. Thus in order to
treat the related Fourier integral operator (corresponding to the estimate of $w^1$),
we distinguish the following three cases according to
the different values of $\delta/\delta_0$:

Case (i) $\delta_0\leq\delta\leq10\cdot 2^{\frac{m+2}{2}}\delta_0$;

Case (ii) $\delta\geq10\cdot 2^{\frac{m+2}{2}}$;

Case (iii) $10\cdot 2^{\frac{m+2}{2}}\delta_0\leq\delta\leq10\cdot 2^{\frac{m+2}{2}}$, if \(\delta_0\leq1\).

Here we point out that for Case (i)- Case (ii) of the wave equation, it is direct to
establish an inequality analogous to \eqref{equ:4.39} (see (3.2) and Section $3$ of \cite{Gls1}). However, for the Tricomi equation, due to the complexity of the fundamental solution,
more delicate and involved techniques from microlocal analysis are required to get the estimate
of $v$.

\subsubsection{Case (i): small \(\delta\)}\label{sec4:large:i}

Note that $\phi_m(1)>0$ and $\delta\phi_m(1)\leq\phi_m(T)$ in the support of \(v\) and
hence $\delta\lesssim \phi_m(T)$ holds. To prove \eqref{equ:4.39}, it suffices to show
\begin{equation}
\phi_m(T)^{\frac{1}{q_0}}\left\| t^{\frac{\alpha}{q_0}}v\right\|_{L^{q_0}(D_{t,x}^{T,\delta})}\leq C\lp t^{-\frac{\alpha}{q_0}}G\rp_{L^{\frac{q_0}{q_0-1}}(\rb)}.
\label{equ:4.40}
\end{equation}
To prove \eqref{equ:4.40} , by the method in \cite[Proposition 3.1]{Gls1} and \cite[Lemma 3.3]{HWY1}, if we write
\begin{equation}\label{equ:4.78}
v(t,x)=\int_0^t(\tilde{H}_{t,s}G)(x)\md s,
\end{equation}
where the operator \(\tilde{H}_{t,s}\) is defined in \eqref{equ:3.28} with the amplitude function \(a(t,s,\xi)\) satisfying \eqref{equ:3.24}, then it suffices to prove

\begin{claim}
\begin{equation}\label{equ:4.41}
\lp t^{\frac{\alpha}{q_0}}(\tilde{H}_{t,s}G)(\cdot)\rp_{L^{q_0}(\mathbb{R}^n)}\leq C|t-s|^{-\frac{2}{q_0}(1+\frac{m}{4})}\lp s^{-\frac{\alpha}{q_0}}G(s,\cdot)\rp_{L^{\frac{q_0}{q_0-1}}(\mathbb{R}^n)}.
\end{equation}
\end{claim}
In fact, with {\bf Claim 3.1}, noting that by the support condition of \(v\) and \(G\), we have for every fixed \(m>0\)
\[t\geq T\geq2\cdot10^{\frac{2}{m+2}}>2\geq s, \]
which yields \(t\geq t-s\gtrsim t\),
we then have
\begin{equation*}
\begin{split}
\lp t^{\frac{\alpha}{q_0}}v\rp_{L^{q_0}([T, 2T]\times\mathbb{R}^n)}\leq&\lp\int_1^2\lp t^{\frac{\alpha}{q_0}}(\tilde{H}_{t,s}G)(\cdot)\rp _{L^{q_0}(\mathbb{R}^n)}\md s\rp _{L^{q_0}_t([T, 2T])} \\
\leq&C\lp\int_1^2|t-s|^{-\frac{2}{q_0}(1+\frac{m}{4})}\lp G(s,\cdot)\rp _{L^{\frac{q_0}{q_0-1}}_x}\md s\rp _{L^{q_0}_t} \\
\leq&C\lp|t|^{-\frac{2}{q_0}(1+\frac{m}{4})}\int_1^2\| G(s,\cdot)\|_{L^{\frac{q_0}{q_0-1}}_x}\md s\rp_{L^{q_0}_t} \\
\leq&C\lc\int _T^{2T}t^{-2(1+\frac{m}{4})}\md t\rc^{\frac{1}{q_0}}\lp s^{-\frac{\alpha}{q_0}}G\rp _{L^{\frac{q_0}{q_0-1}}(\rb)} \ \ \text{(H\"{o}lder's inequality)}\\
\leq&C\phi_m(T)^{-\frac{1}{q_0}}\lp s^{-\frac{\alpha}{q_0}}G\rp _{L^{\frac{q_0}{q_0-1}}(\rb)},
\end{split}
\end{equation*}
which derives \eqref{equ:4.40}.

\begin{proof}[Proof of {\bf Claim 3.1}]
We can make the dyadic decomposition $\tilde{H}_{t,s}G=\ds\sum_{j=-\infty}^{\infty}\tilde{H}_{t,s}^jG$, where
\begin{equation}\label{equ:4.79}
\tilde{H}_{t,s}^jG=\int_{\mathbb{R}^n}e^{i\{x\cdot\xi-[\phi_m(t)-\phi_m(s)]|\xi|\}}
 \chi\lc\frac{|\xi|}{2^j}\rc a(t,s,\xi)\hat{G}(s,\xi) \md\xi
\end{equation}
and the cut-off function \(\chi\) is defined in \eqref{equ:3.23-1}. Since \(1\leq s\leq2\) in the support of  \(G\), \eqref{equ:3.24} implies
\begin{equation}\label{equ:4.80-1}
\begin{split}
	\big| \partial_\xi^\beta a(t,s,\xi)\big|& \leq  C\phi_m(t)^{-\frac{m}{2(m+2)}}|\xi|^{-\frac{m}{2(m+2)}}\phi_m(s)^{-\frac{m}{2(m+2)}}|\xi|^{-\frac{m}{2(m+2)}}|\xi|^{-\frac{2}{m+2}-|\beta|} \\
	&\leq C\phi_m(t)^{-\frac{m}{2(m+2)}}|\xi|^{-1-|\beta|},
\end{split}
\end{equation}
Thus if we further set
\[\tilde{H}_jG(t,s,x)=:t^{\frac{\alpha}{q_0}}(\tilde{H}_{t,s}^jG)(x)\quad\text{with $\lambda_j=2^j$},\]
then an application of FIO theory and stationary phase method yields
\begin{equation*}
\begin{split}
\lp \tilde{H}_jG (t,s,\cdot)\rp _{L^2(\mathbb{R}^n)} \leq C\lambda_j^{-1}t^{\frac{\alpha}{q_0}-\frac{m}{4}}\lp s^{-\frac{\alpha}{q_0}}G(s,\cdot)\rp_{L^2(\mathbb{R}^n)}
\end{split}
\end{equation*}
and
\begin{equation*}
\begin{split}
\qquad\qquad\quad\| \tilde{H}_jG (t,s,\cdot)&\|_{L^\infty(\mathbb{R}^n)}\leq C\lambda_j^{\frac{n+1}{2}-1}t^{\frac{\alpha}{q_0}-\frac{m}{4}}|t-s|^{-\frac{n-1}{2}\cdot\frac{m+2}{2}}\lp
s^{-\frac{\alpha}{q_0}}G(s,\cdot)\rp_{L^1(\mathbb{R}^n)}.
\end{split}
\end{equation*}
By the interpolation and direct computation, we have that for $j\geq0$,
\begin{equation}\label{equ:4.44}
\|\tilde{H}_j(t,s,\cdot)\|_{L^{q_0}(\mathbb{R}^n)}\leq C \lambda^{\frac{n+1}{2}(1-\frac{2}{q_0})-1}t^{\frac{\alpha}{q_0}-\frac{m}{4}}|t-s|^{-\frac{(n-1)(m+2)}{4}(1-\frac{2}{q_0})}
\lp s^{-\frac{\alpha}{q_0}}G(s,\cdot)\rp_{L^{\frac{q_0}{q_0-1}}(\mathbb{R}^n)}.
\end{equation}
Since \(t\geq t-s\gtrsim t\), \eqref{equ:4.44} gives
\begin{equation}
	\begin{split}
		\|\tilde{H}_j(t,s,\cdot)&\|_{L^{q_0}(\mathbb{R}^n)} \\
		&\leq C\lambda^{\frac{n+1}{2}(1-\frac{2}{q_0})-1}|t-s|^{\frac{\alpha}{q_0}-\frac{m}{4}}|t-s|^{-\frac{(n-1)(m+2)}{4}(1-\frac{2}{q_0})}\lp s^{-\frac{\alpha}{q_0}}G(s,\cdot)\rp_{L^{\frac{q_0}{q_0-1}}(\mathbb{R}^n)} \\
		&= C\lambda^{\frac{(n-1)\alpha-mn}{(m+2)n+2\alpha+2}}|t-s|^{-\frac{2}{q_0}(1+\frac{m}{4})}\lp s^{-\frac{\alpha}{q_0}}G(s,\cdot)\rp_{L^{\frac{q_0}{q_0-1}}(\mathbb{R}^n)}.
	\end{split}
\end{equation}
Since \(-1<\frac{(n-1)\alpha-mn}{(m+2)n+2\alpha+2}<0\) provided \(\alpha< m\cdot\frac{n}{n-1}\), if we denote
\(\tilde{H}_+G (t,s,x)\equiv\ds \sum_{j\geq0}\tilde{H}_j(t,s,x),\)
 then
 \begin{equation}\label{equ:4.44-1}
 	\| \tilde{H}_+G(t,s,\cdot)\|_{L^{q_0}(\mathbb{R}^n)}\leq C|t-s|^{-\frac{2}{q_0}(1+\frac{m}{4})}
 	\lp s^{-\frac{\alpha}{q_0}}G(s,\cdot)\rp _{L^{\frac{q_0}{q_0-1}}(\mathbb{R}^n)}.
 \end{equation}

For $j<0$, let
\[\tilde{H}_-G (t,s,x)\equiv\ds \sum_{j<0}\tilde{H}_j(t,s,x)
=\int_{|\xi|\leq1}e^{i\{x\cdot\xi-[\phi_m(t)-\phi_m(s)]|\xi|\}}t^{\frac{\alpha}{q_0}}a(t,s,\xi)
\hat{G}(s,\xi) d\xi.\]
Then it follows from Plancherel's identity that
\begin{equation*}
\begin{split}
\lp \tilde{H}_-G (t,s,\cdot)\rp_{L^2(\mathbb{R}^n)}
\leq C\lc\int_{|\xi|\leq1}\left||\xi|^{-\f{2}{m+2}}t^{\frac{\alpha}{q_0}}\big(1+\phi_m(t)|\xi|\big)^{-\frac{m}{2(m+2)}}\hat{G}(s,\xi)\right|^2 \md\xi\rc^{\frac{1}{2}}.
\end{split}
\end{equation*}
Note that for \(1\leq s\leq2\)
\[\left|\hat{G}(s,\xi)\right|=\left|\int_{\mathbb{R}^n}e^{-iy\cdot\xi}G(s,y)\md y\right|
=\left|\int_{|y|\leq\phi_m(2)}e^{-iy\cdot\xi}G(s,y)\md y\right|
\leq C\lp s^{-\frac{\alpha}{q_0}}G(s,\cdot)\rp _{L^2(\mathbb{R}^n)},\]
then we can compute in the polar coordinate that
\begin{equation*}
\begin{split}
\bigg(\int_{|\xi|\leq1}\big||\xi|^{-\f{2}{m+2}}t^{\frac{\alpha}{q_0}}&(1+\phi_m(t)|\xi|)^{-\frac{m}{2(m+2)}}\big|^2 \md\xi\bigg)^{\frac{1}{2}} \\
&\le Ct^{\frac{\alpha}{q_0}}\lc\int_0^1r^{n-1-\frac{m+4}{m+2}}\phi_m(t)^{-\frac{m}{m+2}
}\md r\rc^{\frac{1}{2}}\le Ct^{\frac{\alpha}{q_0}-\frac{m}{4}},
\end{split}
\end{equation*}
here we have used the fact of $n-1-\frac{m+4}{m+2}\geq-\frac{2}{m+2}>-1$ for $n\geq2$ and $m>0$. Thus by the condition \(t\geq t-s\gtrsim t\) one has
\begin{equation*}
\lp \tilde{H}_-G (t,s,\cdot)\rp_{L^2(\mathbb{R}^n)}\leq Ct^{\frac{\alpha}{q_0}-\frac{m}{4}}\lp G(s,\cdot)\rp_{L^2(\mathbb{R}^n)}
\end{equation*}
Similarly, we have by stationary phase method,
\begin{equation*}
\lp \tilde{H}_-G (t,s,\cdot)\rp_{L^\infty(\mathbb{R}^n)}
\leq C|t-s|^{\frac{\alpha}{q_0}-\frac{m}{4}-\frac{n-1}{2}\cdot\frac{m+2}{2}}\lp G(s,\cdot)\rp_{L^1(\mathbb{R}^n)}.
\end{equation*}
Using interpolation again, we get
\begin{equation}\label{equ:4.48}
\lp \tilde{H}_-G (t,s,\cdot)\rp_{L^{q_0}(\mathbb{R}^n)}\leq C|t-s|^{-\frac{2}{q_0}(1+\frac{m}{4})}\lp G(s,\cdot)\rp_{L^{\frac{q_0}{q_0-1}}(\mathbb{R}^n)}.
\end{equation}
Then by Littlewood-Paley theory,  \eqref{equ:4.44-1} and \eqref{equ:4.48}, {\bf Claim 3.1}
is established.
\end{proof}

\subsubsection{Case (ii): large \(\delta\)}\label{sec4:large:ii}

In this case, one has $\phi_m(t)-|x|\geq\delta\phi_m(1)\gtrsim 10\phi_m(2)$. As in \eqref{equ:4.78} and \eqref{equ:4.79}, we can write
\[v=\sum_{j=-\infty}^{\infty}v_j=\sum_{j=-\infty}^{\infty}\int_0^t\int_{\Bbb R^n} K_j(t,x;s,y)G(s,y)dyds,\]
where
\begin{equation}
 K_j(t,x;s,y)=\int_{\mathbb{R}^n}e^{i\{(x-y)\cdot\xi-[\phi_m(t)-\phi_m(s)]|\xi|\}}
\chi\lc\frac{|\xi|}{2^j}\rc a(t,s,\xi)G(s,y) d\xi,
\label{equ:4.50}
\end{equation}
moreover, as in (3.41) of \cite{HWY1}, the kernel \(K_j\) satisfies for $\lambda_j=2^j$ and any $N\in\Bbb R^+$,
\begin{equation}
\begin{split}
|K_j(t,x;s,y)|\leq &C_N\lambda_j^{\frac{n+1}{2}-\frac{2}{m+2}}\big(|\phi_m(t)-\phi_m(s)|
+\lambda_j^{-1}\big)^{-\frac{n-1}{2}}\big(1+\phi_m(t)\lambda_j\big)^{-\frac{m}{2(m+2)}} \\
&\times \Big(1+\lambda_j\big||\phi_m(t)-\phi_m(s)|-|x-y|\big|\Big)^{-N}.
\end{split}
\label{equ:4.51}
\end{equation}
Denote $D_{s,y}^{\delta_0}=\{(s,y): 1\leq s\leq2, \delta_0\phi_m(1)\leq\phi_m(s)-|y|\leq2\delta_0\phi_m(1)\}$. By H\"{o}lder's inequality
and the support condtiton of $G(s,y)$ with respect to the variable $(s,y)$, we arrive at
\[
\begin{split}
	|t^{\frac{\alpha}{q_0}}v_j|\leq
& \Big\|t^{\frac{\alpha}{q_0}} K_j(t,x;s,y)\big(\phi^2_m(t)-|x|^2\big)^{\frac{1}{q_0}}s^{\frac{\alpha}{q_0}}\Big\|_{L^{q_0}(D_{s,y}^{\delta_0})} \\
	& \times\Big\|\big(\phi^2_m(t)-|x|^2\big)^{-\frac{1}{q_0}}s^{-\frac{\alpha}{q_0}}G(s,y)
\Big\|_{L^{\frac{q_0}{q_0-1}}(D_{s,y}^{\delta_0})}.
\end{split}
\]
In addition, by applying the support condition of $G(s,y)$, it is easy to check
\[\phi_m(t)-\phi_m(s)-|x-y|\geq C(\phi_m(t)-|x|), \quad \phi_m(t)-\phi_m(s)+|x-y|\sim\phi_m(t).\]

Based on this observation, let $N=\frac{n}{2}-\frac{1}{m+2}$ in \eqref{equ:4.51}, we then have
\begin{equation}\label{equ:4.52}
\begin{split}
&\Big\| t^{\frac{\alpha}{q_0}}K_j(t,x;s,y)\big(\phi^2_m(t)-|x|^2\big)^{\frac{1}{q_0}}s^{-\frac{\alpha}{q_0}}\Big\|_{L^{q_0}(D_{s,y}^{\delta_0})} \\
&\leq C\bigg(\iint_{D_{s,y}^{\delta_0}}\bigg\{t^{\frac{\alpha}{q_0}}\lambda_j^{\frac{n+1}{2}-\frac{2}{m+2}}\big(|\phi_m(t)-\phi_m(s)|
+\lambda_j^{-1}\big)^{-\frac{n-1}{2}}(1+\phi_m(t)\lambda_j)^{-\frac{m}{2(m+2)}} \\
&\qquad \times \Big(1+\lambda_j\big||\phi_m(t)-\phi_m(s)|-|x-y|\big|\Big)^{\frac{1}{m+2}-\frac{n}{2}}\bigg\}^{q_0}
\phi_m(t)(\phi_m(t)-|x|)\md y\md s\bigg)^{\frac{1}{q_0}} \\
&\leq C\phi_m(t)^{-\frac{n-1}{2}+\frac{1}{q_0}+\frac{2\alpha}{(m+2)q_0} -\frac{m}{2(m+2)}}\big(\phi_m(t)-|x|\big)^{\frac{1}{m+2}-\frac{n}{2}+\frac{1}{q_0}}
\Big(\iint_{D_{s,y}^{\delta_0}}\md y\md s\Big)^{\frac{1}{q_0}} \\
&\leq C\delta_0^{\frac{1}{q_0}}\phi_m(t)^{-\frac{n}{2}+\frac{1}{m+2}+\frac{1}{q_0}
+\frac{2\alpha}{(m+2)q_0}}\big(\phi_m(t)-|x|\big)^{\frac{1}{m+2}-\frac{n}{2}+\frac{1}{q_0}}
\end{split}
\end{equation}
and
\begin{equation}\label{equ:4.53}
\begin{split}
\Big(\iint_{D_{s,y}^{\delta_0}}\big\{(\phi^2_m(t)-|x|^2)^{-\frac{1}{q_0}}&s^{-\frac{\alpha}{q_0}}G(s,y)\big\}^{\frac{q_0}{q_0-1}}\md y\md s\Big)^{\frac{q_0-1}{q_0}} \\
&\leq C\big(\delta\phi_m(T)\big)^{-\frac{1}{q_0}}\| s^{-\frac{\alpha}{q_0}}G\|_{L^{\frac{q_0}{q_0-1}}(\ra)} \\
\end{split}
\end{equation}
On the other hand,
\begin{equation}
\begin{split}
&\lp\phi_m(t)^{-\frac{n}{2}+\frac{1}{m+2}+\frac{1}{q_0}
+\frac{2\alpha}{(m+2)q_0}}\big(\phi_m(t)-|x|\big)^{\frac{1}{m+2}-\frac{n}{2}
+\frac{1}{q_0}}\rp_{L^{q_0}([T, 2T]\times\mathbb{R}^n)}\\
&=C_n\lc\int_T^{2T}\phi_m(t)^{q_0(\frac{n}{2}-\frac{1}{m+2})+1+\frac{2\alpha}{m+2}}\int_0^{\phi_m(t)-10\phi_m(2)}
\big(\phi_m(t)-r\big)^{q_0(\frac{1}{m+2}-\frac{n}{2})+1}r^{n-1}\md r\md t\rc^{\frac{1}{q_0}} \\
&\leq C\lc\int_T^{2T}\phi_m(t)^{-\frac{2}{m+2}}\md t\rc^{\frac{1}{q_0}}\leq C.
\end{split}
\label{equ:4.54}
\end{equation}

Therefore, combining \eqref{equ:4.52}-\eqref{equ:4.54} gives

\begin{align}\label{equ:4.103}
\| v_j\|_{L^{q_0}(D_{t,x}^{T,\delta})}\leq&C\delta_0^{\frac{1}{q_0}}(\delta\phi_m(T))^{-\frac{1}{q_0}}\| s^{-\frac{\alpha}{q_0}}G\|_{L^{\frac{q_0}{q_0-1}}(\ra)} \no\\
\leq &C\delta_0^{\frac{1}{q_0}}\phi_m(T)^{-\frac{1}{q_0}+\frac{\nu}{2}}\delta^{-\frac{1}{q_0}-\frac{\nu}{2}}
\| s^{-\frac{\alpha}{q_0}}G\|_{L^{\frac{q_0}{q_0-1}}(\ra)},
\end{align}
here we have used the fact of  $\delta\lesssim\phi_m(T)$ due to $2\delta\phi_m(1)\leq\phi_m(T)$. Then
\eqref{equ:4.103} together with Lemma \ref{lem:a5} yields estimate \eqref{equ:4.39} in Case (ii).

\subsubsection{Case (iii): medium \(\delta\)}\label{sec4:large:iii}

Motivated by the ideas in Section 3 of \cite{Gls1}, we shall decompose
the related Fourier integral operator in the expression of $v$
into a high frequency part and a low frequency part, then the two parts are
treated with different techniques respectively.
First, by \eqref{equ:3.23}, the solution \(v\) of \eqref{equ:4.104} can be expressed as
\begin{equation}
v=\int_{\rb}\int_{\Bbb R^n}e^{i\{(x-y)\cdot\xi-[\phi_m(t)-\phi_m(s)]|\xi|\}}a(t,s,\xi)G(s,y) \md\xi\md y\md s
\label{equ:4.6}
\end{equation}
by \eqref{equ:3.24} we have for $\kappa\in \Bbb N_0^n$,
\begin{equation}\label{equ:4.7}
\begin{split}
	\big| \partial_\xi^\kappa a(t,s,\xi)\big|& \leq C\big(1+\phi_m(t)|\xi|\big)^{-\frac{m}{2(m+2)}}|\xi|^{-\frac{2}{m+2}-|\kappa|}, \\
\end{split}
\end{equation}
\begin{equation}\label{equ:4.8}
	\big| \partial_\xi^\kappa a(t,s,\xi)\big|\leq C\phi_m(t)^{-\frac{m}{2(m+2)}}|\xi|^{-1-|\kappa|}.
\end{equation}

Set $\tau=\phi_m(s)-|y|$. Applying H\"{o}lder's inequality, one then has that
\begin{equation}\label{equ:4.55}
\begin{split}
&\left|t^{\frac{\alpha}{q_0}}v\right|=\left|\int_{\phi_m(1)}^{\phi_m(2)}\int_{\mathbb{R}^n}\int_{\mathbb{R}^n}e^{i\{(x-y)\cdot\xi-[\phi_m(t)-\phi_m(s)]|\xi|\}}a(t,s,\xi)G(s,y)\md\xi\md y \frac{1}{\phi_m(s)} \md\phi_m(s)\right| \\
&\lesssim \int_{\delta_0}^{2\delta_0}\left|\int_{\mathbb{R}^n}\int_{\mathbb{R}^n}e^{i\{(x-y)\cdot\xi-[\phi_m(t)
-\tau-|y|]|\xi|\}}a(t,\phi_m^{-1}(\tau+|y|),\xi)G(\phi_m^{-1}(\tau+|y|),y)\md\xi\md y \right|\md\tau\\
&\le C\delta_0^{\frac{1}{q_0}}\bigg(\int_{\delta_0}^{2\delta_0}\Big|\int_{\Bbb R^n}\int_{\Bbb R^n} e^{i\{(x-y)\cdot\xi-[\phi_m(t)-\tau-|y|]|\xi|\}}t^{\frac{\alpha}{q_0}}a(t,\phi_m^{-1}(\tau+|y|),\xi) \\
&\qquad\qquad\qquad\qquad\qquad\qquad \times G\big(\phi_m^{-1}(\tau+|y|),y\big)\md\xi\md y
\Big|^{\frac{q_0}{q_0-1}}\md \tau\bigg)^{\frac{q_0-1}{q_0}}.
\end{split}
\end{equation}

Next note that by \eqref{equ:4.30},
\(\phi_m(t)\geq\phi_m(T)\geq10\phi_m(2)\),
this together with $\tau<\phi_m(s)<\phi_m(2)$ yields $\phi_m(t)\geq\phi_m(t)-\tau>\frac{1}{2}\phi_m(t)$.
Thus we can replace $\phi_m(t)-\tau$ with $\phi_m(t)$ in \eqref{equ:4.55} and consider
\[\mathcal{T}g(t,x)=\int_{\Bbb R^n}\int_{\Bbb R^n} e^{i\{(x-y)\cdot\xi-[\phi_m(t)-|y|]|\xi|\}}
b(t,\xi) g(y)\md\xi\md y.\]
The estimate of $\|\mathcal{T}g\|_{L_x^{q_0}}$ is divided into different parts according to the range of  \(\alpha\) and \(|\xi|\).
\paragraph{Case (iii-1) \(-1<\alpha<0\)} In this case, \eqref{equ:4.7} implies that for $\kappa\in \Bbb N_0^n$,
\[\partial_\xi^\kappa b(t,\xi)\leq C\big(1+\phi_m(t)|\xi|\big)^{-\frac{m}{2(m+2)}}|\xi|^{-\frac{2}{m+2}-|\kappa|} \]
which motivate us to consider for \(z\in\mathbb{C}\)
\begin{equation}
\begin{split}
(\mathcal{T}_zg)(t,x)=&\left(z-\frac{(m+2)n+2+2\alpha}{(m+2)(\alpha+2)}\right)e^{z^2}  \\
&\times \int_{\mathbb{R}^n}\int_{\mathbb{R}^n}
e^{i\{(x-y)\cdot\xi-[\phi_m(t)-|y|]|\xi|\}}t^{\frac{\alpha}{q_0}}\big(1+\phi_m(t)|\xi|\big)^{-\frac{m}{2(m+2)}}g(y)\frac{\md\xi}{|\xi|^z}\md y,
\end{split}
\label{equ:H.1}
\end{equation}
where $\phi_m(t)\geq10\phi_m(2)-\phi_m(2)\geq9\phi_m(2)$ and $\delta<10\phi_m(2)$, then by \cite[Lemma A.2.]{HWY4}, \eqref{equ:4.39} follows from
\begin{equation}
\begin{split}
	\| (\mathcal{T}_{z}g)(t,\cdot) &\|_{L^{q_0}(\{x:\delta\phi_m(1)\leq\phi_m(t)-|x|\leq2\delta\phi_m(1)\})} \\
	&\leq C\phi_m(t)^{\frac{\nu}{2}-\frac{m+4}{q_0(m+2)}}\delta^{-\frac{\nu}{2}
-\frac{1}{q_0}}
\| g\|_{L^{\frac{q_0}{q_0-1}}(\mathbb{R}^n)}, \quad Rez=\frac{2}{m+2} .
\end{split}
\label{equ:4.57}
\end{equation}
For clearer statement on \eqref{equ:4.57}, we shall replace $\frac{\nu}{2}$ by $\frac{\nu}{q_0}$ and rewrite
\eqref{equ:4.57} as
\begin{equation}
\begin{split}
	\| (\mathcal{T}_{z}g)(t,\cdot) &\|_{L^{q_0}(\{x:\delta\phi_m(1)\leq\phi_m(t)-|x|\leq2\delta\phi_m(1)\})} \\
	&\leq C\phi_m(t)^{\frac{\nu}{q_0}-\frac{m+4}{q_0(m+2)}}\delta^{-\frac{\nu}{q_0}
-\frac{1}{q_0}}
\| g\|_{L^{\frac{q_0}{q_0-1}}(\mathbb{R}^n)}, \quad Rez=\frac{2}{m+2} .
\end{split}
\label{equ:4.56}
\end{equation}

Next we focus on the proofs of \eqref{equ:4.56} and \eqref{equ:4.57}. We shall use the complex interpolation
method to establish \eqref{equ:4.56},

then \eqref{equ:4.56} would be a consequence of
\begin{equation}
\| (\mathcal{T}_zg)(t,\cdot)\|_{L^\infty(\mathbb{R}^n)}\leq Ct^{\frac{\alpha}{q_0}}\phi_m(t)^{-\frac{n-1}{2}-\frac{m}{2(m+2)}}\| g\|_{L^1(\mathbb{R}^n)}, Rez=\frac{(m+2)n+2+2\alpha}{(m+2)(\alpha+2)},
\label{equ:4.59}
\end{equation}
and
\begin{equation}
\| (\mathcal{T}_zg)(t,\cdot)\|_{L^2(\mathbb{R}^n)}\leq Ct^{\frac{\alpha}{q_0}}\phi_m(t)^{-\frac{m}{2(m+2)}}
(\phi_m(t)^{\nu}\delta^{-(\nu+1)})^{\frac{1}{m+2}}\| g\|_{L^2(\mathbb{R}^n)}, \quad Rez=0.
\label{equ:4.60}
\end{equation}

In fact, the interpolation between \eqref{equ:4.60} and \eqref{equ:4.59} gives
\begin{equation*}
\begin{split}
	&\| (\mathcal{T}_{z}g)(t,\cdot)\|_{L^{q_0}(\{x:\delta\phi_m(1)\leq\phi_m(t)-|x|\leq2\delta\phi_m(1)\})} \\
	& \leq C\phi_m(t)^{\frac{2\alpha}{(m+2)q_0}-\frac{m}{2(m+2)} +(-\frac{n-1}{2})(1-\frac{2}{q_0})+ \frac{2\nu}{q_0(m+2)}}\delta^{-\frac{2(\nu+1)}{(m+2)q_0}
}\| g\|_{L^{\frac{q_0}{q_0-1}}(\mathbb{R}^n)} \\
    & \leq C\phi_m(t)^{\frac{\nu}{q_0}-\frac{m+4}{q_0(m+2)}}\delta^{-\frac{2(\nu+1)}{(m+2)q_0}}\| g\|_{L^{\frac{q_0}{q_0-1}}(\mathbb{R}^n)}.
\end{split}
\end{equation*}
By $\delta\leq10\phi_m(2)$, \eqref{equ:4.56} is established.

We now prove \eqref{equ:4.59} by the stationary phase method.
To this end, for
$z=\frac{(m+2)n+2+2\alpha}{(m+2)(\alpha+2)}+i\theta$ with $\theta\in\Bbb R$, we have for all \(-1<\alpha<0\),
\[\frac{(m+2)n+2+2\alpha}{(m+2)(\alpha+2)}>\frac{(m+2)n+2}{2(m+2)},\]
thus there exists \(\sigma>0\) such that
\[\frac{(m+2)n+2+2\alpha}{(m+2)(\alpha+2)}-\sigma >\frac{(m+2)n+2}{2(m+2)},\]
hence,
\begin{equation}
\begin{split}
&\bigg| \theta e^{-\theta^2}\int_{\mathbb{R}^n}\int_{|\xi|\geq1}e^{i[(x-y)\cdot\xi-(\phi_m(t)-|y|)|\xi|]}t^{\frac{\alpha}{q_0}}\big(1+\phi_m(t)|\xi|\big)^{-\frac{m}{2(m+2)}}g(y)\frac{\md\xi}{|\xi|^z}\md y \bigg| \\
&\leq C\big(\phi_m(t)-\phi_m(2)\big)^{-\frac{n-1}{2}}\phi_m(t)^{-\frac{m}{2(m+2)}}
t^{\frac{\alpha}{q_0}}\int_{|\xi|\geq1} |\xi|^{-\frac{n-1}{2}-\frac{m}{2(m+2)}}
|\xi|^{-\frac{(m+2)n+2}{2(m+2)}-\sigma}\md \xi \| g\|_{L^1(\mathbb{R}^n)} \\
&\leq C\phi_m(t)^{-\frac{n-1}{2}-\frac{m}{2(m+2)}}t^{\frac{\alpha}{q_0}}\int_1^\infty r^{-1-\sigma}\md r\| g\|_{L^1(\mathbb{R}^n)} \\
& \leq C\phi_m(t)^{-\frac{n}{2}+\frac{1}{m+2}}t^{\frac{\alpha}{q_0}}
\| g\|_{L^1(\mathbb{R}^n)}.
\end{split}
\label{equ:4.62}
\end{equation}

On the other hand, we have that for $|\xi|\leq1$,
\begin{equation}\label{equ:4.64}
\begin{split}
&\left| \theta e^{-\theta^2}\int_{|\xi|\leq1}\int_{\mathbb{R}^n}e^{i[(x-y)\cdot\xi-(\phi_m(t)-|y|)|\xi|]} t^{\frac{\alpha}{q_0}}\big(1+\phi_m(t)|\xi|\big)^{-\frac{m}{2(m+2)}}g(y)\frac{\md\xi}{|\xi|^z}\md y \right| \\
&\leq C\| g\|_{L^1(\mathbb{R}^n)} \int_{|\xi|\leq1}t^{\frac{\alpha}{q_0}}\big(1+\phi_m(t)|\xi|\big)^{-\frac{n-1}{2}-\frac{m}{2(m+2)}}|\xi|^{-\frac{(m+2)n+2+2\alpha}{(m+2)(\alpha+2)}}\md\xi\\
&\leq C\| g\|_{L^1(\mathbb{R}^n)} t^{\frac{\alpha}{q_0}}\int_0^1(1+\phi_m(t)r)^{-\frac{n}{2}+\frac{1}{m+2}}
r^{-\frac{(m+2)n+2+2\alpha}{(m+2)(\alpha+2)}}r^{n-1}\md r,
\end{split}
\end{equation}
here we have noted the fact of
 \[n-1-\frac{(m+2)n+2+2\alpha}{(m+2)(\alpha+2)}>n-1-\frac{(m+2)n}{m+2}=-1, \quad \text{for} -1<\alpha<0\]
 and $n\geq2$, thus the integral in last line of \eqref{equ:4.64} is convergent.
 In order to give a precise estimate
to \eqref{equ:4.64}, denoting \(\sigma=n-\frac{(m+2)n+2+2\alpha}{(m+2)(\alpha+2)}\), then the integral in \eqref{equ:4.64} can be controlled by
\begin{equation}\label{equ:4.64-1}
	\begin{split}
		&\int_0^1(1+\phi_m(t)r)^{-\frac{n}{2}+\frac{1}{m+2}}r^{-1+\sigma}\md r \\
		&=\sigma^{-1}\phi_m(t)^{-\frac{n}{2}+\frac{1}{m+2}}+\frac{(m+2)n-2}{2\sigma(m+2)}\int_0^1\left(1+\phi_m(t)r\right)^{-\frac{n}{2}+\frac{1}{m+2}-1}
		r^\sigma\md r
	\end{split}
\end{equation}
For the last integral in \eqref{equ:4.64-1}, note that
\[
\begin{split}
	\left(1+\phi_m(t)r\right)^{-\frac{n}{2}+\frac{1}{m+2}-1}r^\sigma
&=(1+\phi_m(t)r)^{-\frac{n}{2}+\frac{1}{m+2}}r^{-1+\sigma}\frac{r}{1+\phi_m(t)r} \\
&\leq\phi_m(t)^{-1}(1+\phi_m(t)r)^{-\frac{n-1}{2}-\frac{m}{2(m+2)}}r^{-1+\sigma},
\end{split}
 \]
for every fixed \((n,m,\alpha)\), \(\frac{(m+2)n+2}{2\sigma(m+2)}\) is a given positive constant, therefore without loss of generalirity, we can assume \(\phi_m(t)\geq\phi_m(T)\geq\frac{(m+2)n+2}{\sigma(m+2)}\), otherwise the estimate \eqref{equ:4.39} can be established as the relatively small time case in Section 3.3.2. Then we have
\begin{equation*}
\begin{split}
&\int_0^1(1+\phi_m(t)r)^{-\frac{n}{2}+\frac{1}{m+2}}r^{-1+\sigma}\md r\leq\frac{\phi_m(t)^{-\frac{n}{2}+\frac{1}{m+2}}}{\sigma}
+\frac{1}{2}\int_0^1\left(1+\phi_m(t)r\right)^{-\frac{n}{2}+\frac{1}{m+2}}r^{-1+\sigma}\md r,
\end{split}
\end{equation*}
which implies
\begin{equation}\label{equ:4.65}
	\begin{split}
		&\left| \theta e^{-\theta^2}\int_{|\xi|\leq1}\int_{\mathbb{R}^n}e^{i\{(x-y)\cdot\xi-[\phi_m(t)-|y|]|\xi|\}} \big(1+\phi_m(t)|\xi|\big)^{-\frac{m}{2(m+2)}}g(y)\frac{\md\xi}{|\xi|^z}\md y \right| \\
		&\leq\frac{C}{\sigma}\phi_m(t)^{-\frac{n}{2}+\frac{1}{m+2}}t^{\frac{\alpha}{q_0}}\| g\|_{L^1(\mathbb{R}^n)}
	\end{split}
\end{equation}
Thus combining \eqref{equ:4.62} and \eqref{equ:4.65} yields \eqref{equ:4.59} with a constant $C>0$ depends on $n$, $m$ and $\alpha$.

To get \eqref{equ:4.60}, the low frequencies and high frequencies will be treated separately.
As in \cite{Gls1}, we shall use Sobolev trace theorem to handle the low frequency part.
More specifically, we first introduce a function $\rho\in C^\infty(\mathbb{R}^n)$ such that
\begin{equation*}
\rho(\xi)=
\left\{ \enspace
\begin{aligned}
1, \quad &|\xi|\geq2, \\
0, \quad &|\xi|\leq1.
\end{aligned}
\right.
\end{equation*}

For $\varrho=1+\nu$, let
$$(\mathcal{T}_zg)(t,x)=(R_zg)(t,x)+(S_zg)(t,x),$$
where
\begin{align}\label{equ:4.115}
(R_zg)(t,x)=&\Big(z-\frac{(m+2)n+2+2\alpha}{(m+2)(\alpha+2)}\Big)e^{z^2}t^{\frac{\alpha}{q_0}}
\int_{\mathbb{R}^n}\int_{\frac{1}{2}\phi_m(1)\leq|y|\leq\phi_m(2)}e^{i\{(x-y)\cdot\xi-[\phi_m(t)-|y|]|\xi|\}} \no\\
&\qquad\qquad \qquad\times \big(1+\phi_m(t)|\xi|\big)^{-\frac{m}{2(m+2)}}
\Big(1-\rho\big(\phi_m(t)^{1-\varrho}\delta^\varrho\xi\big)\Big)
g(y)\frac{\md\xi}{|\xi|^z}\md y,\no\\
(S_zg)(t,x)=&\Big(z-\frac{(m+2)n+2+2\alpha}{(m+2)(\alpha+2)}\Big)e^{z^2}t^{\frac{\alpha}{q_0}}
\int_{\mathbb{R}^n}\int_{\frac{1}{2}\phi_m(1)\leq|y|\leq\phi_m(2)}e^{i\{(x-y)\cdot\xi-[\phi_m(t)-|y|]|\xi|\}}  \no\\
&\qquad\qquad \qquad \qquad\times \big(1+\phi_m(t)|\xi|\big)^{-\frac{m}{2(m+2)}}\rho\big(\phi_m(t)^{1-\varrho}\delta^\varrho\xi\big)g(y)\frac{\md\xi}{|\xi|^z}\md y.
\end{align}
Note that for \(Rez=0\), the integral in \(R_zg\) and \(S_zg\) are the same as the integral in \cite[(4-44)]{HWY4}, thus \cite[(4-45)-(4-46)]{HWY4} imply

\begin{equation}
\| (R_zg)(t,\cdot)\|_{L^2(\mathbb{R}^n)}\leq C\phi_m(t)^{\frac{2\alpha}{(m+2)q_0}-\frac{m}{2(m+2)}}(\phi_m(t)^{\varrho-1}\delta^{-\varrho})^{\frac{1}{m+2}}\| g\|_{L^2(\mathbb{R}^n)}
\quad Rez=0,
\label{equ:4.66}
\end{equation}
and
\begin{equation}
\begin{split}
\| S_zg(t,\cdot)&\|_{L^2(\{x:\delta\leq\phi_m(t)-|x|\leq2\delta\})} \\
\leq & C\phi_m(t)^{\frac{2\alpha}{(m+2)q_0}-\frac{m}{2(m+2)}}(\phi_m(t)^{\varrho-1}\delta^{-\varrho})^{\frac{1}{m+2}}\| g\|_{L^2(\mathbb{R}^n)}, \qquad Rez=0.
\end{split}
\label{equ:4.67}
\end{equation}
Note that \eqref{equ:4.66} together with \eqref{equ:4.67} yields \eqref{equ:4.60}. Interpolation between \eqref{equ:4.59} and \eqref{equ:4.60} implies \eqref{equ:4.56} and further

\begin{equation}
\begin{split}
	\| (\mathcal{T}g)(t,\cdot) &\|_{L^{q_0}(\{x:\delta\phi_m(1)\leq\phi_m(t)-|x|\leq2\delta\phi_m(1)\})}\lesssim \phi_m(t)^{\frac{\nu}{2}-\frac{m+4}{q_0(m+2)}}\delta^{-\frac{\nu}{2}
-\frac{1}{q_0}}
\| g\|_{L^{\frac{q_0}{q_0-1}}(\mathbb{R}^n)}, \end{split}
\label{equ:4.57-1}
\end{equation}
with \(-1<\alpha<0\).

\paragraph{Case (iii-2) \(0\leq\alpha\leq m\)} If \(0<|\xi|<1\), then the analysis of estimating  $\|\mathcal{T}g\|_{L_x^{q_0}}$ is the same as Case (iii-1) via the formula of  \(\mathcal{T}_zg\) in \eqref{equ:H.1}, and we omit the detials. However, for the case \(|\xi|\geq1\), \eqref{equ:4.8} implies that
\[\partial_\xi^\kappa b(t,\xi)\leq C\phi_m(t)^{-\frac{m}{2(m+2)}}|\xi|^{-1-|\kappa|} \]
which motivate us to consider for \(z\in\mathbb{C}\)
\begin{equation}
\begin{split}
(\tilde{\mathcal{T}}_zg)(t,x)=&\left(z-\frac{(m+2)n+2+2\alpha}{2(\alpha +2)}\right)e^{z^2}t^{\frac{\alpha}{q_0}}\phi_m(t)^{-\frac{m}{2(m+2)}}  \\
&\times \int_{\mathbb{R}^n}\int_{\mathbb{R}^n}
e^{i\{(x-y)\cdot\xi-[\phi_m(t)-|y|]|\xi|\}} g(y)\frac{\md\xi}{|\xi|^z}\md y,
\end{split}
\label{equ:H.2}
\end{equation}
We now prove \eqref{equ:4.59} by the stationary phase method.
To this end, for
$z=\frac{(m+2)n+2+2\alpha}{2(\alpha+2)}+i\theta$ with $\theta\in\Bbb R$, we have for all \(0\leq\alpha\leq m\),
\[-\frac{(m+2)n+2+2\alpha}{2(\alpha +2)}<-\frac{(m+2)n+2+2m}{2(m+2)}<-\frac{(m+2)n+2}{2(m+2)},\]
thus there exists \(\sigma>0\) such that
\[-\frac{(m+2)n+2+2\alpha}{2(\alpha +2)}<-\frac{(m+2)n+2}{2(m+2)}-\sigma,\]
hence
\begin{equation}
\begin{split}
&\left| \theta e^{-\theta^2}\int_{\mathbb{R}^n}\int_{|\xi|\geq1}e^{i\{(x-y)\cdot\xi-[\phi_m(t)-|y|]|\xi|\}} t^{\frac{\alpha}{q_0}}\big(1+\phi_m(t)|\xi|\big)^{-\frac{m}{2(m+2)}}g(y)\frac{\md\xi}{|\xi|^z}\md y \right| \\
&\leq C\big(\phi_m(t)-\phi_m(2)\big)^{-\frac{n-1}{2}}\phi_m(t)^{-\frac{m}{2(m+2)}}
t^{\frac{\alpha}{q_0}}\int_{|\xi|\geq1} |\xi|^{-\frac{n-1}{2}-\frac{m}{2(m+2)}}
|\xi|^{-\frac{(m+2)n+2}{2(m+2)}-\sigma}\md \xi \| g\|_{L^1(\mathbb{R}^n)} \\
&\leq C\phi_m(t)^{-\frac{n-1}{2}-\frac{m}{2(m+2)}}t^{\frac{\alpha}{q_0}}\int_1^\infty r^{-1-\sigma}\md r\| g\|_{L^1(\mathbb{R}^n)} \\
& \leq C\phi_m(t)^{-\frac{n-1}{2}-\frac{m}{2(m+2)}}t^{\frac{\alpha}{q_0}}
\| g\|_{L^1(\mathbb{R}^n)}.
\end{split}
\label{equ:4.62-1}
\end{equation}
Thus
\begin{equation}
\| (\tilde{\mathcal{T}}_zg)(t,\cdot)\|_{L^\infty(\mathbb{R}^n)}\leq Ct^{\frac{\alpha}{q_0}}\phi_m(t)^{-\frac{n-1}{2}-\frac{m}{2(m+2)}}\| g\|_{L^1(\mathbb{R}^n)},  Rez=\frac{(m+2)n+2+2\alpha}{2(\alpha +2)}.
\label{equ:4.59-1}
\end{equation}
While for the \(L^2\) estimate, like the case \(-1<\alpha<0\), one need to handle
\begin{align}\label{equ:4.116}
(\tilde{R}_zg)&(t,x)=\Big(z-\frac{(m+2)n+2+2\alpha}{2(\alpha +2)}\Big)e^{z^2}
t^{\frac{\alpha}{q_0}}\phi_m(t)^{-\frac{m}{2(m+2)}}\no\\
&\times
\int_{|\xi|\geq1}\int_{\frac{1}{2}\phi_m(1)\leq|y|\leq\phi_m(2)}e^{i\{(x-y)\cdot\xi-[\phi_m(t)-|y|]|\xi|\}} \Big(1-\rho\big(\phi_m(t)^{1-\varrho}\delta^\varrho\xi\big)\Big)
g(y)\frac{\md\xi}{|\xi|^z}\md y,\no\\
(\tilde{S}_zg)&(t,x)=\Big(z-\frac{(m+2)n+2+2\alpha}{2(\alpha +2)}\Big)e^{z^2}
t^{\frac{\alpha}{q_0}}\phi_m(t)^{-\frac{m}{2(m+2)}}\no\\
&\times \int_{|\xi|\geq1}\int_{\frac{1}{2}\phi_m(1)\leq|y|\leq\phi_m(2)}e^{i\{(x-y)\cdot\xi-[\phi_m(t)-|y|]|\xi|\}} \rho\big(\phi_m(t)^{1-\varrho}\delta^\varrho\xi\big)g(y)\frac{\md\xi}{|\xi|^z}\md y.
\end{align}

It follows from Lemma~\ref{lem:a3} and  Lemma~\ref{lem:a4} in Appendix that
\begin{equation}
\| (\tilde{R}_zg)(t,\cdot)\|_{L^2(\mathbb{R}^n)}\leq C\phi_m(t)^{\frac{2\alpha}{(m+2)q_0}-\frac{m}{2(m+2)}}(\phi_m(t)^{\varrho-1}\delta^{-\varrho})^{\frac{1}{2}}\| g\|_{L^2(\mathbb{R}^n)}
\quad Rez=0,
\label{equ:4.66-1}
\end{equation}
and
\begin{equation}
\begin{split}
\| \tilde{S}_zg(t,\cdot)&\|_{L^2(\{x:\delta\leq\phi_m(t)-|x|\leq2\delta\})} \\
\leq & C\phi_m(t)^{\frac{2\alpha}{(m+2)q_0}-\frac{m}{2(m+2)}}(\phi_m(t)^{\varrho-1}\delta^{-\varrho})^{\frac{1}{2}}\| g\|_{L^2(\mathbb{R}^n)}, \qquad Rez=0.
\end{split}
\label{equ:4.67-1}
\end{equation}
\eqref{equ:4.66-1} together with \eqref{equ:4.67-1} yield
\begin{equation}
\| (\tilde{\mathcal{T}}_zg)(t,\cdot)\|_{L^2(\mathbb{R}^n)}\leq Ct^{\frac{\alpha}{q_0}}\phi_m(t)^{-\frac{m}{2(m+2)}}
(\phi_m(t)^{\nu}\delta^{-(\nu+1)})^{\frac{1}{2}}\| g\|_{L^2(\mathbb{R}^n)}, \quad Rez=0.
\label{equ:4.60-1}
\end{equation}
Interpolation between \eqref{equ:4.59-1} with \eqref{equ:4.60-1} gives
\[
\begin{split}
	\| (\tilde{\mathcal{T}}_{z}g)(t,\cdot) &\|_{L^{q_0}(\{x:\delta\phi_m(1)\leq\phi_m(t)-|x|\leq2\delta\phi_m(1)\})} \\
	&\leq C\phi_m(t)^{\frac{\nu}{q_0}-\frac{m+4}{q_0(m+2)}}\delta^{-\frac{\nu}{q_0}
-\frac{1}{q_0}}
\| g\|_{L^{\frac{q_0}{q_0-1}}(\mathbb{R}^n)}, \quad Rez=1,
\end{split}
\]
then \eqref{equ:4.57-1} can be established for \(0<\alpha\leq m\) and the proof for the medium \(\delta\) case is finished.

\textbf{Collecting all the analysis above in Case (i)- Case (iii), \eqref{equ:4.39} is proved for the relatively large times.}

\vspace{4mm}

\textbf{Combining the results in Section 3.3 and Section 3.4, we have obtained \eqref{equ:4.34}, therefore \eqref{equ:3.3} is established.}

\section{The proof of Theorem 1.3 at the end point \(\mathbf{q=2}\)}

\subsection{Simplifications for the end point estimate}
\label{sect: H}

In this section we establish another endpoint estimate \eqref{equ:3.4.1} for  \(q=2\) in Theorem~\ref{thm:inhomogeneous estimate}.
Suppose that $w$ solves \eqref{Y-3}, where $F\equiv0$ if $\phi_m(t)-|x|<1$. Then by Theorem 2.1 of \cite{Yag2}, we have
\[\| w(t,\cdot)\|_{L^2(\mathbb{R}^n)}\leq Ct\int_{T_0}^t\| F(s,\cdot)\|_{L^2(\mathbb{R}^n)}\md s,\]
which yields that for ${T_0}\leq t\leq5$ the \(L^2\) estimate,
\[\| w\|_{L^2([{T_0},5]\times\mathbb{R}^n)}\leq C\| F\|_{L^2([{T_0},5]\times\mathbb{R}^n)}.\]

Note that $\phi_m(t)-|x|$ is bounded from below and above when ${T_0}\leq t\leq5$, hence for any $\nu>0$,

\begin{equation}
\begin{split}
\Big\|\big(\phi^2_m(t)-|x|^2\big)^{-\frac{1}{2}+\frac{m-\alpha}{m+2}-\nu}t^{\frac{\alpha}{2}}w\Big\|_{L^2([{T_0},5]\times\mathbb{R}^n)}\leq C \Big\|\big(\phi^2_m(t)-|x|^2\big)^{\frac{1}{2}+\nu}t^{-\frac{\alpha}{2}}F\Big\|_{L^2(\ra)}.
\end{split}
\label{E1}
\end{equation}

Next we suppose $T\geq5$. As in Section 3.2, we make the decomposition \(F=F^0+F^1\) with \(F^0\) defined in \eqref{equ:4.30}, and split $w$ as $w=w^0+w^1$, where for $j=0,1$, \((\partial_t^2-t^m \Delta)w^j =F^j\) with zero data. Then in order to prove \eqref{equ:3.4.1}, it suffices to show that for $j=0$, $1$,

\begin{equation}
\begin{split}
\Big\|\big(\phi^2_m(t)-|x|^2\big)^{-\frac{1}{2}+\frac{m-\alpha}{m+2}-\nu}&t^{\frac{\alpha}{2}}w^j\Big\|_{L^2(\{(t,x):T\leq t\leq 2T\})} \\
&\leq C\phi_m(T)^{-\frac{\nu}{4}}\Big\|\big(\phi^2_m(t)-|x|^2\big)^{\frac{1}{2}+\nu}t^{-\frac{\alpha}{2}}F^j\Big\|_{L^2(\ra)}.
\end{split}
\label{equ:5.2}
\end{equation}

Note that by the analogous treatment on $w^j$ as in \eqref{equ:4.34}-\eqref{equ:4.39},
\eqref{equ:5.2} will follow from
\begin{equation}
\begin{split}
	&\phi_m(T)^{-\frac{1}{2}+\frac{m-\alpha}{m+2}-\frac{\nu}{2}}\delta^{-\frac{1}{2}+\frac{m-\alpha}{m+2}+\frac{\nu}{2}}\lp t^{\frac{\alpha}{2}}v\rp_{L^2(D_{t,x}^{T,\delta})}
	\leq C\delta_0^{\frac{1}{2}}\lp s^{-\frac{\alpha}{2}}G\rp_{L^2(\rb)},
\end{split}
\label{equ:5.7}
\end{equation}
where \(D_{t,x}^{T,\delta}\) was defined in \eqref{equ:D-tx} and $\operatorname{supp}G\subseteq D_{s,y}^{\delta_0}=\{(s,y): 1\leq s\leq2, \delta_0\phi_m(1)\leq\phi_m(s)-|y|\leq2\delta_0\phi_m(1)\}$, and $\delta\geq\delta_0$.
Next we focus on the proof of \eqref{equ:5.7}. For technical reason, we will first treat the ``relatively large time" case and establish \(L^2\) estimate for \(v^1\). Then the estimate for \(v^0\) follows with similar idea but easier computation.

\subsection{Estimate for large times}
Note that $\phi_m(T)\geq10\phi_m(2)$
holds for $(t,x)\in \operatorname{supp} v^1$. As in Section \ref{sec4},
we shall deal with the estimates according to the different scales of \(\delta\).

\paragraph{The case of $\mathbf{\delta\geq10\phi_m(2)}$}

As in Subsection \ref{sec4:large:ii}, we shall use the pointwise estimate to handle the case of $\phi_m(t)-|x|\geq\delta\geq10\phi_m(2)$.
We now write
\[v=\sum_{j=-\infty}^{\infty}v_j=\sum_{j=-\infty}^{\infty}\iint_{D_{s,y}^{\delta_0}} K_j(t,x;s,y)G(s,y)dyds,\]
where
\begin{equation*}
 K_j(t,x;s,y)=\int_{\mathbb{R}^n}e^{i\{(x-y)\cdot\xi-[\phi_m(t)-\phi_m(s)]|\xi|\}}
\chi\Big(\frac{|\xi|}{2^j}\Big)a(t,s,\xi)\hat{G}(s,\xi) \md\xi.
\end{equation*}

Let \(\epsilon>0\) be a sufficiently small constant, then by \eqref{equ:4.51} and H\"{o}lder's inequality, we arrive at
\begin{equation}
	\begin{split}
		|t^{\frac{\alpha}{2}}v_j|\leq &\lp t^{\frac{\alpha}{2}}K_j(t,x;s,y)(\phi_m(t)+|x|)^{\frac{1}{2}}(\phi_m(t)-|x|)^{\frac{1}{2}-\frac{1}{m+2}-\epsilon}s^{\frac{\alpha}{2}}\rp_{L^2_{s,y}(D_{s,y}^{\delta_0})} \\
		&\times\lp(\phi_m(t)+|x|)^{-\frac{1}{2}}(\phi_m(t)-|x|)^{-\frac{1}{2}+\frac{1}{m+2}+\epsilon}s^{-\frac{\alpha}{2}}G(s,y)\rp_{L^2_{s,y}(D_{s,y}^{\delta_0})}.
	\end{split}
\end{equation}
Taking $N=\f{n}{2}-\f{1}{m+2}$ in \eqref{equ:4.51} and repeating the computations of \eqref{equ:4.52} and \eqref{equ:4.53}, we have
\begin{equation*}
\begin{split}
&\lp t^{\frac{\alpha}{2}} K_j(t,x;s,y)(\phi_m(t)+|x|)^{\frac{1}{2}}(\phi_m(t)-|x|)^{\frac{1}{2}-\frac{1}{m+2}-\epsilon}s^{\frac{\alpha}{2}}\rp_{L^2_{s,y}(D_{s,y}^{\delta_0})} \\
&\leq C\delta_0^{\frac{1}{2}}\phi_m(t)^{-\frac{n-1}{2}+\frac{1+\alpha}{m+2}}(\phi_m(t)-|x|)^{-\frac{n}{2}+\frac{1}{2}-\epsilon}
\end{split}
\end{equation*}
and
\[
\begin{split}
	&\left(\iint_{D_{s,y}^{\delta_0}}\big\{(\phi_m(t)+|x|)^{-\frac{1}{2}}(\phi_m(t)-|x|)^{-\frac{1}{2}+\frac{1}{m+2}+\epsilon}s^{-\frac{\alpha}{2}}G(s,y)\big\}^2\md y\md s\right)^{\frac{1}{2}} \\
	&\leq C\big(\delta\phi_m(T)\big)^{-\frac{1}{2}}\delta^{\frac{1}{m+2}+\epsilon}
\lp s^{-\frac{\alpha}{2}}G\rp_{L^2(\rb)}.
\end{split}
\]

In addition, by $-\frac{n-1}{2}-\epsilon <-\frac{1}{2}$ for $n\geq2$, direct computation yields
\begin{equation*}
\begin{split}
\Big\|&\phi_m(t)^{-\frac{n-1}{2}+\frac{1+\alpha}{m+2}}\big(\phi_m(t)-|x|\big)^{-\frac{n-1}{2}-\epsilon}\Big\|_{L_{t,x}^2(D_{t,x}^{T,\delta})} \\
&\leq C_n\lc\int_T^{2T}\phi_m(t)^{-(n-1)+\frac{2(1+\alpha)}{m+2}}\int_0^{\phi_m(t)-10\phi_m(2)}
(\phi_m(t)-r)^{-2\epsilon -(n-1)}r^{n-1}\md r\md t\rc^{\frac{1}{2}} \\
&\leq C\lc\int_T^{2T}\phi_m(t)^{\frac{2(1+\alpha)}{m+2}}\md t\rc^{\frac{1}{2}}\leq CT^{1+\frac{\alpha}{2}}.
\end{split}
\end{equation*}
Thus we obtain
\begin{equation}
\begin{split}
&\phi_m(T)^{-\frac{1}{2}+\frac{m-\alpha}{m+2}-\frac{\nu}{2}}\delta^{-\frac{1}{2}+\frac{m-\alpha}{m+2}+\frac{\nu}{2}}\lp t^{\frac{\alpha}{2}}v\rp_{L^2(D_{t,x}^{T,\delta})} \\
&\leq\phi_m(T)^{-\frac{1}{2}+\frac{m-\alpha}{m+2}-\frac{\nu}{2}}
\delta^{-\frac{1}{2}+\frac{m-\alpha}{m+2}+\frac{\nu}{2}}\delta_0^{\frac{1}{2}}
\big(\delta\phi_m(T)\big)^{-\frac{1}{2}}\delta^{\frac{1}{m+2}+\epsilon}
T^{1+\frac{\alpha}{2}}\lp t^{-\frac{\alpha}{2}}G\rp_{L^2(\rb)} \\
&\leq C\delta_0^{\frac{1}{2}}\left(\frac{\delta}{\phi_m(t)} \right)^{\frac{\nu}{2}} \delta^{-\frac{\alpha+1}{m+2}+\epsilon} \lp t^{-\frac{\alpha}{2}}G\rp_{L^2(\rb)}.
\end{split}
\label{E2}
\end{equation}
By the condition \(\alpha>-1\), we have \(-\frac{\alpha+1}{m+2}<0\), thus there exists a \(\epsilon>0\) small enough such that
\[-\frac{\alpha+1}{m+2}+\epsilon <0,\]
then by \(\phi_m(t)\gtrsim \delta\geq10\phi_m(2)\), \eqref{equ:5.7} is proved.

\paragraph{The case of $\mathbf{\delta_0\leq\delta\leq10\phi_m(2)}$}\label{sec4:w0:small}

Next we study \eqref{equ:5.7} under the condition $\phi_m(t)-|x|\leq10\phi_m(2)$.
At first, we claim that under certain restrictions on the variable $\xi$,
this situation can be treated as in the proof of \eqref{equ:4.60} in Section 3.
Indeed, recalling \eqref{D4}, \(v\) can be written as:
\begin{equation*}
\begin{split}
v=&\int_{\rb}\int_{\Bbb R^n}e^{i\{(x-y)\cdot\xi-[\phi_m(t)-\phi_m(s)]|\xi|\}}a(t,s,\xi)G(s,y) \md\xi\md y\md s.
\end{split}
\end{equation*}
where \(a(t,s,\xi)\) satisfying \eqref{equ:3.24}. Noting that $t\geq s\gtrsim 1$, then we can assume
\begin{equation}\label{equ:v-L2}
	v=\int_1^2\int_{\mathbb{R}^n}\int_{\mathbb{R}^n}e^{i\{(x-y)\cdot\xi-[\phi_m(t)-\phi_m(s)]|\xi|\}}
\phi_m(t)^{-\frac{m}{2(m+2)}}|\xi|^{-1}G(s,y)\md y\md\xi \md s.
\end{equation}
As in the proof of \eqref{equ:4.60}, we again split $v$ into a low
frequency part and a high frequency part respectively.
To this end, we choose a function $\beta\in C_0^\infty(\mathbb{R}^n)$ satisfying $\beta=1$ near the origin
such that $t^{\frac{\alpha}{2}}v=v_0+v_1$, where
\begin{equation*}
\begin{split}
v_1&=\int_1^2\int_{\mathbb{R}^n}\int_{\mathbb{R}^n}e^{i\{(x-y)\cdot\xi-[\phi_m(t)-\phi_m(s)]|\xi|\}}
\phi_m(t)^{-\frac{m}{2(m+2)}}t^{\frac{\alpha}{2}}\frac{1-\beta(\delta\xi)}{|\xi|}G(s,y)\md y\md\xi \md s.\\
\end{split}
\end{equation*}
If we set \(\phi_m(s)=|y|+\tau\) and use H\"{o}lder's inequality as in \eqref{equ:4.55}, then
\begin{equation*}
\begin{split}
|v_1|\leq& C\delta_0^{\f{1}{2}}\bigg(\int_{\delta_0}^{2\delta_0}\bigg|\phi_m(t)^{-\frac{m}{2(m+2)}}t^{\frac{\alpha}{2}} \\
&\times\int_{\mathbb{R}^n}\int_{\mathbb{R}^n}e^{i\{(x-y)\cdot\xi-[\phi_m(t)-|y|-\tau]|\xi|\}}\frac{1-\beta(\delta\xi)}{|\xi|}G(\phi_m^{-1}(|y|+\tau) ,y)\md y\md\xi\bigg|^2
\md \tau\bigg)^{\f{1}{2}} \\
=&:C\delta_0^{\f{1}{2}}\bigg(\int_{\delta_0}^{2\delta_0}|T_1(t,\tau,\cdot)|^2\md \tau\bigg)^{\f{1}{2}}.
\end{split}
\end{equation*}
Note $\frac{1-\beta(\delta\xi)}{|\xi|}=O(\delta)$. Then the expression of $T_1$ is similar to \eqref{equ:H.2} with $Rez=0$, with this observation we apply the method of \eqref{equ:4.60-1} to get
\[\| T_1(t,\tau,\cdot)\|_{L^2(\mathbb{R}^n)}\leq C\big(\phi_m(t)-\tau\big)^{\frac{\nu}{2}-\frac{m}{2(m+2)}}t^{\frac{\alpha}{2}}
\delta^{-\frac{\nu+1}{2}+1}\| G(s,\cdot)\|_{L^2(\mathbb{R}^n)},\]
which derives
\begin{equation}\label{Y-5}
\| v_1\|_{L^2}\leq C\delta_0^{\frac{1}{2}}\delta^{-\frac{\nu+1}{2}+1}\phi_m(T)^{\frac{\nu}{2}-\frac{m-2}{2(m+2)}}T^{\frac{\alpha}{2}}\| G\|_{L^2(\rb)}.
\end{equation}
Due to the condition $\delta\lesssim 10\phi_m(2)$, the estimate \eqref{equ:5.7} for $v_1$ follows immediately from \eqref{Y-5}.

We now estimate $v_0$. At first, one notes that
\begin{equation}\label{equ:5.15}
\begin{split}
&\Big|\int_{|\xi|\leq1}e^{i\{(x-y)\cdot\xi-[\phi_m(t)-\phi_m(s)]|\xi|\}}\phi_m(t)^{-\frac{m}{2(m+2)}}t^{\frac{\alpha}{2}}\frac{\beta(\delta\xi)}{|\xi|}\md\xi\Big| \\
&\le C\big(1+\big|\phi_m(t)-\phi_m(s)\big|\big)^{-\frac{n-1}{2}}\phi_m(t)^{-\frac{m}{2(m+2)}}t^{\frac{\alpha}{2}} \\
&\le C\big(1+|x-y|\big)^{-\frac{n-1}{2}}\phi_m(t)^{-\frac{m}{2(m+2)}}t^{\frac{\alpha}{2}}.
\end{split}
\end{equation}
In the last step of \eqref{equ:5.15} we have used the fact $\phi_m(t)-\phi_m(s)\geq|x-y|$ for any $(s,y)\in \operatorname{supp}F$ and $(t,x)\in \operatorname{supp}w$. This condition can be derived from
the formula in Theorem 2.4 of \cite{Yag3}, see Section 5B2 in our former work \cite{HWY4} for details.

Note that the corresponding inequality \eqref{equ:5.7} holds when we replace $v$ by
\[v_{01}=\int_1^2\int_{\mathbb{R}^n}\int_{|\xi|\leq1}e^{i\{(x-y)\cdot\xi-[\phi_m(t)-\phi_m(s)]|\xi|\}}\phi_m(t)^{-\frac{m}{2(m+2)}}
\frac{\beta(\delta\xi)}{|\xi|}G(s,y)\md y\md\xi \md s.\]

It follows from method of stationary phase and direct computation that
\begin{equation}\label{equ:5.16}
\begin{split}
 &\lp t^{\frac{\alpha}{2}}v_{01}\rp_{L^2(D_{t,x}^{T,\delta})} \\
&\leq\lp\iiint_{|\xi|\leq1}e^{i\{(x-y)\cdot\xi-[\phi_m(t)-\phi_m(s)]|\xi|\}}t^{\frac{\alpha}{2}}\phi_m(t)^{-\frac{m}{2(m+2)}}\frac{\beta(\delta\xi)}{|\xi|}\md\xi
G(s,y)\md y\md s\rp_{L^2_{t,x}(D_{t,x}^{T,\delta})} \\
&\leq C\lp\iint\big(1+\big|x-y\big|\big)^{-\frac{n-1}{2}}t^{\frac{\alpha}{2}}\phi_m(t)^{-\frac{m}{2(m+2)}}G(s,y)\md y\md s\rp_{L^2_{t,x}(D_{t,x}^{T,\delta})} \\
&\leq C\phi_m(T)^{-\frac{m}{2(m+2)}}T^{\frac{\alpha}{2}}\lp\lp\big(1+\big|x-y\big|\big)^{-\frac{n-1}{2}}s^{\frac{\alpha}{2}}\rp_{L^2_{s,y}(D_{s,y}^{\delta_0})}\lp s^{-\frac{\alpha}{2}}G\rp_{L^2_{s,y}(D_{s,y}^{\delta_0})}\rp_{L^2_{t,x}(D_{t,x}^{T,\delta})} \\
&\lesssim \phi_m(T)^{\frac{\alpha}{m+2}-\frac{m}{2(m+2)}}\| s^{-\frac{\alpha}{2}}G\|_{L^2_{s,y}(D_{s,y}^{\delta_0})} \\
&\quad\times\bigg\|\bigg(\int_T^{2T}\int_{\delta\phi_m(1)\leq\phi_m(t)-|x|\leq2\delta\phi_m(1)}
\frac{\md x\md t}{\big(1+|x-y|\big)^{n-1}}\bigg)^{\frac{1}{2}}\bigg\|_{L^2_{s,y}(D_{s,y}^{\delta_0})}.
\end{split}
\end{equation}
By $|y|\leq\phi_m(2)$, we see that $\frac{1}{2}|x|\leq|x-y|\leq2|x|$ holds if $|x|\geq2\phi_m(2)$. On the other hand, if $|x|<2\phi_m(2)$,
then the integral with respect to the variable $x$ in last line of  in \eqref{equ:5.16}
must be convergent and can be controlled by $\delta$. This yields
\begin{equation*}
\begin{split}
\bigg\|\bigg(\int_T^{2T}&\int_{\delta\phi_m(1)\leq\phi_m(t)-|x|\leq2\delta\phi_m(1)}
\frac{\md x\md t}{\big(1+|x-y|\big)^{n-1}}\bigg)^{\frac{1}{2}}\bigg\|_{L^2_{s,y}(D_{s,y}^{\delta_0})}
  \\
& \leq C\bigg\|\bigg(\int_T^{2T}\int_{\phi_m(t)-2\delta}^{\phi_m(t)-\delta}\md r\md t\bigg)^{\frac{1}{2}}\bigg\|_{L^2_{s,y}(D_{s,y}^{\delta_0})}\leq C(\delta_0\delta T)^{\frac{1}{2}},
\end{split}
\end{equation*}
which implies that the left side of \eqref{equ:5.7} can be controlled by
\begin{equation}
\begin{split}
&\phi_m(T)^{-\frac{1}{2}+\frac{m-\alpha}{m+2}-\frac{\nu}{2}}\delta^{-\frac{1}{2}+\frac{m-\alpha}{m+2}
+\frac{\nu}{2}}
\phi_m(T)^{\frac{\alpha}{m+2}-\frac{m}{2(m+2)}}(\delta_0\delta T)^{\frac{1}{2}}\lp  s^{-\frac{\alpha}{2}}G\rp_{L^2_{s,y}(D_{s,y}^{\delta_0})} \\
&\leq C\delta^{\frac{m-\alpha}{m+2}}\delta_0^{\frac{1}{2}}\lp  s^{-\frac{\alpha}{2}}G\rp_{L^2_{s,y}(D_{s,y}^{\delta_0})}\leq C\delta_0^{\frac{1}{2}}\lp  s^{-\frac{\alpha}{2}}G\rp_{L^2(\rb)}.
\end{split}
\label{equ:A}
\end{equation}
In the last inequality we have used \(\delta\lesssim 10\phi_m(2)\) and \(\alpha\leq m\).

Consequently, the proof of \eqref{equ:5.7}  will be completed if we could show that
\begin{equation}
\begin{split}
&\phi_m(T)^{-\frac{1}{2}+\frac{m-\alpha}{m+2}-\frac{\nu}{2}}\delta^{-\frac{1}{2}+\frac{m-\alpha}{m+2}
+\frac{\nu}{2}}\lp t^{\frac{\alpha}{2}}v_{02}\rp_{L^2(D_{t,x}^{T,\delta})}\leq C\delta_0^{\frac{1}{2}}\lp s^{-\frac{\alpha}{2}}G\rp_{L^2(\ra)},
\end{split}
\label{equ:5.10}
\end{equation}
where
\[v_{02}=\int_1^2\int_{\mathbb{R}^n}\int_{|\xi|\geq1}e^{i\{(x-y)\cdot\xi-[\phi_m(t)-\phi_m(s)]|\xi|\}}\phi_m(t)^{-\frac{m}{2(m+2)}}
\frac{\beta(\delta\xi)}{|\xi|}G(s,y)\md y\md\xi \md s.\]

The first step in proving \eqref{equ:5.10} is to notice that
\begin{equation*}
\begin{split}
&\lp t^{\frac{\alpha}{2}}v_{02}\rp_{L^2(D_{t,x}^{T,\delta})}\leq\lp \int\parallel\check{T}G\parallel_{L^2(\{x:\delta\phi_m(1)\leq\phi_m(t)-|x|\leq2\delta\phi_m(1)\})}ds
\rp_{L^2(\{t:\frac{T}{2}\leq t\leq T\})},
\end{split}
\end{equation*}
where
\[\check{T}G=\int\int_{|\xi|\geq1}e^{i\{(x-y)\cdot\xi-[\phi_m(t)-\phi_m(s)]|\xi|\}}\phi_m(t)^{-\frac{m}{2(m+2)}}t^{\frac{\alpha}{2}}
\frac{\beta(\delta\xi)}{|\xi|}G(s,y)\md y\md\xi.\]
To estimate $\parallel\check{T}G\parallel_{L^2(\{x:\delta\phi_m(1)\leq\phi_m(t)-|x|\leq2\delta\phi_m(1)\})}$,
it follows from Lemma \ref{lem:a1} and direct computation
\begin{equation}
\begin{split}
&\phi_m(T)^{-\frac{1}{2}+\frac{m-\alpha}{m+2}-\frac{\nu}{2}}\delta^{-\frac{1}{2}+\frac{m-\alpha}{m+2}
+\frac{\nu}{2}}\lp t^{\frac{\alpha}{2}}v_{02}\rp_{L^2(D_{t,x}^{T,\delta})} \\
&\leq C\phi_m(T)^{-\frac{1}{2}+\frac{m-\alpha}{m+2}-\frac{\nu}{2}}\delta^{-\frac{1}{2}+\frac{m-\alpha}{m+2}+\frac{\nu}{2}}\delta^{\frac{1}{2}}T^{\frac{1}{2}}
\phi_m(T)^{\frac{\alpha}{m+2}-\frac{m}{2(m+2)}} \\
&\qquad \times \lc\sum_{j=0}^{\infty}2^{\frac{j}{2}}
\lp\iint e^{i\{-y\cdot\xi+\phi_m(s))|\xi|\}}\frac{\beta(\delta\xi)}{|\xi|}G(s,y)\md y\md s\rp_{L^2(2^j\leq|\xi|\leq2^{j+1})}\rc \\
&\leq C\delta^{\frac{m-\alpha}{m+2}}\left(\frac{\delta}{\phi_m(T)} \right)^{\frac{\nu}{2}} \lc\sum_{j=0}^{\infty}\Big\|\iiint_{2^j\leq|\xi|\leq2^{j+1}} e^{i\{(x-y)\cdot\xi+\phi_m(s))|\xi|\}}\frac{\beta(\delta\xi)}{|\xi|^{\frac{1}{2}}}G(s,y)\md y\md\xi \md s\Big\|_{L^2_x}\rc.
\end{split}
\label{equ:5.11}
\end{equation}
By applying H\"{o}lder's inequality as in \eqref{equ:4.55}, we get
\begin{equation}
\begin{split}
&\Big\|\iiint_{2^j\leq|\xi|\leq2^{j+1}} e^{i\{(x-y)\cdot\xi+\phi_m(s))|\xi|\}}\frac{\beta(\delta\xi)}{|\xi|^{\frac{1}{2}}}G(s,y)\md y \md\xi \md s\Big\|_{L^2_x} \\
&\leq C\delta_0^{\frac{1}{2}}\lc\iint\Big|\iint_{2^j\leq|\xi|\leq2^{j+1}}e^{i\{(x-y)\cdot\xi+(|y|+\tau)|\xi|\}}
\frac{\beta(\delta\xi)}{|\xi|^{\frac{1}{2}}}
G(\phi_m^{-1}(|y|+\tau),y)\md y\md\xi\Big|^2\md x\md\tau\rc^{\frac{1}{2}}.
\end{split}
\label{equ:5.12}
\end{equation}
On the other hand, an application of Lemma 3.2 in \cite{Gls1} yields that
for each fixed $j\ge 0$,
\begin{equation}
\begin{split}
&\Big\|\iint_{2^j\leq|\xi|\leq2^{j+1}}e^{i\{(x-y)\cdot\xi+(|y|+\tau)|\xi|\}}\frac{\beta(\delta\xi)}{|\xi|^{\frac{1}{2}}}
G(\phi_m^{-1}(|y|+\tau),y)\md y\md\xi\Big\|_{L^2_{\tau,x}} \\
&\leq C\Big(\iint_{\phi_m(1)\leq|y|+\tau\leq \phi_m(2)}\left|G\big(\phi_m^{-1}(|y|+\tau),y\big)\right|^2\md y\md \tau\Big)^{\frac{1}{2}}\\
&\leq C\Big(\int_{\mathbb{R}^n}\int_1^2|G(s,y)|^2s^{\frac{m}{2}}\md s\md y\Big)^{\frac{1}{2}} \\
&\leq C\Big(\int_{\mathbb{R}^n}\int_1^2|G(s,y)|^2\md s\md y\Big)^{\frac{1}{2}}\leq C\lp s^{-\frac{\alpha}{2}}G\rp_{L^2{(\ra)}}. \\
\end{split}
\label{equ:5.13}
\end{equation}
In addition, in the support of $\beta(\delta\xi)$, one has $2^j\delta\leq|\xi|\delta\leq C$, which derives
\begin{equation}
\begin{split}
&j\leq C(1+|\ln{\delta}|).
\end{split}
\label{equ:5.14}
\end{equation}
Substituting \eqref{equ:5.13} and \eqref{equ:5.14} into \eqref{equ:5.12} and further \eqref{equ:5.11}, we finally get

\begin{align*}
&\phi_m(T)^{-\frac{1}{2}+\frac{m-\alpha}{m+2}-\frac{\nu}{2}}\delta^{-\frac{1}{2}+\frac{m-\alpha}{m+2}+\frac{\nu}{2}}\lp t^{\frac{\alpha}{2}}v_{02}\rp_{L^2(D_{t,x}^{T,\delta})} \\
&\leq C\delta^{\frac{m-\alpha}{m+2}}(1+|\ln{\delta}|)\delta^{\frac{\nu}{2}} \phi_m(t)^{-\frac{\nu}{2}}\delta_0^{\frac{1}{2}}\lp s^{-\frac{\alpha}{2}}G\rp_{L^2{(\ra)}}\leq C\delta_0^{\frac{1}{2}}\lp s^{-\frac{\alpha}{2}}G\rp_{L^2{(\ra)}}.
\end{align*}

\subsection{Estimate for small times}
Now it remains to prove \eqref{equ:5.7} for \(\phi_m(T)\leq10\phi_m(4)\), where
we have \(\delta_0\leq\delta\leq10\cdot4^{\frac{m+2}{2}}\). Our task is reduced to prove
\begin{equation}
\delta^{-\frac{1}{2}+\frac{m-\alpha}{m+2}+\frac{\nu}{2}}\lp t^{\frac{\alpha}{2}}v\rp_{L^2(D_{t,x}^{T,\delta})}\leq C\delta_0^{\frac{1}{2}}\lp s^{-\frac{\alpha}{2}}G\rp_{L^2(\ra)}.
\label{equ:5.17}
\end{equation}

Since \(1\lesssim s\leq t\leq4\cdot10^{\frac{2}{m+2}}\), as in \eqref{equ:v-L2} , we can write
\[v=\int_0^t\int_{\mathbb{R}^n}\int_{\mathbb{R}^n}e^{i\{(x-y)\cdot\xi-[\phi_m(t)-\phi_m(s)]|\xi|\}}
|\xi|^{-1}G(s,y)\md y\md\xi \md s,\]
and split $v$ into a low frequency  part and a high frequency part respectively. We choose a function $\beta\in C_0^\infty(\mathbb{R}^n)$ satisfying $\beta=1$ near the origin such that $v=v_0+v_1$, where
\begin{equation*}
\begin{split}
v_1&=\int_0^t\int_{\mathbb{R}^n}\int_{\mathbb{R}^n}e^{i\{(x-y)\cdot\xi-[\phi_m(t)-\phi_m(s)]|\xi|\}}
\frac{1-\beta(\delta\xi)}{|\xi|}G(s,y)\md y\md\xi \md s. \\
\end{split}
\end{equation*}

If we set \(\phi_m(s)=|y|+\tau\) and use H\"{o}lder's inequality as in \eqref{equ:4.55}, then
\begin{equation*}
\begin{split}
|v_1|&\lesssim \delta_0^{\f{1}{2}}\bigg(\int_{\delta_0}^{2\delta_0}\bigg|\int_{\mathbb{R}^n}\int_{\mathbb{R}^n}e^{i\{(x-y)\cdot\xi-[\phi_m(t)-\phi_m(s)]|\xi|\}}
\frac{1-\beta(\delta\xi)}{|\xi|}G(\phi_m^{-1}(|y|+\tau),y)\md y\md\xi\bigg|^2\md \tau\bigg)^{\f{1}{2}}\\
&=:C\delta_0^{\f{1}{2}}\bigg(\int_{\delta_0}^{2\delta_0}|\bar{T_1}(t,\tau,\cdot)|^2\md \tau\bigg)^{\f{1}{2}}.
\end{split}
\end{equation*}

Note $\frac{1-\beta(\delta\xi)}{|\xi|}=O(\delta)$. Then the expression of $v_1$ is similar to \eqref{equ:H.2} with $Rez=0$. Consequently we can apply the method of \eqref{equ:4.60-1} to get
\[\|\bar{T_1}(t,\tau,\cdot)\|_{L^2(\mathbb{R}^n)}\leq C\big(\phi_m(t)-\tau\big)^{\frac{\nu}{2}}
\delta^{-\frac{\nu+1}{2}+1}\|G(s,\cdot)\|_{L^2(\mathbb{R}^n)},\]
which derives
\begin{equation}\label{Y-4}
\lp t^{\frac{\alpha}{2}}v_1\rp_{L^2}\leq C\delta_0^{\frac{1}{2}}\delta^{-\frac{\nu+1}{2}+1}\phi_m(T)^{\frac{\nu}{2}+\frac{\alpha}{m+2}}\|G\|_{L^2(\ra)}.
\end{equation}
Due to $\delta\lesssim \phi_m(T)\lesssim 10\phi_m(4)$ and \(\phi_m(T)\geq\phi_m(1)\), the estimate \eqref{equ:5.17} for $v_1$ is an immediately consequence of \eqref{Y-4}.

We now estimate $v_0$. At first, similarly to \eqref{equ:5.15}, we have
\begin{equation*}
\begin{split}
\bigg|\int_{|\xi|\leq1}e^{i\{(x-y)\cdot\xi-[\phi_m(t)-\phi_m(s)]|\xi|\}}\frac{\beta(\delta\xi)}{|\xi|}\md\xi\bigg|
\le C\big(1+\big|x-y\big|\big)^{-\frac{n-1}{2}}.
\end{split}
\end{equation*}
Thus the corresponding inequality \eqref{equ:5.7} holds if we replace $v$ by
\[v_{01}=\int_1^2\int_{\mathbb{R}^n}\int_{|\xi|\leq1}e^{i\{(x-y)\cdot\xi-[\phi_m(t)-\phi_m(s)]|\xi|\}}
\frac{\beta(\delta\xi)}{|\xi|}G(s,y)\md y\md\xi \md s.\]

As in \eqref{equ:5.16} one has
\begin{equation*}
\begin{split}
&\lp t^{\frac{\alpha}{2}}v_{01}\rp_{L^2(D_{t,x}^{T,\delta})} \\
&\leq\bigg\|t^{\frac{\alpha}{2}}\iiint_{|\xi|\leq1}e^{i\{(x-y)\cdot\xi-[\phi_m(t)-\phi_m(s)]|\xi|\}}\frac{\beta(\delta\xi)}{|\xi|}\md\xi
G(s,y)\md y\md s\bigg\|_{L^2(D_{t,x}^{T,\delta})} \\
&\leq C\lp s^{-\frac{\alpha}{2}}G\rp_{L^2}\bigg\|\Big(\int_T^{2T}\int_{\delta\leq\phi_m(t)-|x|\leq2\delta}
(1+\big|x-y\big|)^{-(n-1)}\md x\md t\Big)^{\frac{1}{2}}\bigg\|_{L^2_{s,y}(D_{s,y}^{\delta_0})}.
\end{split}
\end{equation*}

Note that in this case of \(|x|\leq\phi_m(t)\leq10\phi_m(2)\), direct computation yields
\begin{equation*}
\begin{split}
\bigg\|&\Big(\int_T^{2T}\int_{\delta\leq\phi_m(t)-|x|\leq2\delta}
\big(1+\big|x-y\big|\big)^{-(n-1)}\md x\md t\Big)^{\frac{1}{2}}\bigg\|_{L^2_{s,y}(D_{s,y}^{\delta_0})} \\
&\leq C\lp\Big(\int_T^{2T}\int_{\phi_m(t)-2\delta}^{\phi_m(t)-\delta}\md r\md t\Big)^{\frac{1}{2}}\rp_{L^2_{s,y}(D_{s,y}^{\delta_0})}\leq C(\delta_0\delta T)^{\frac{1}{2}},
\end{split}
\end{equation*}
which implies that the left side of \eqref{equ:5.17} is bounded by
\begin{equation}
\begin{split}
\delta^{-\frac{1}{2}+\frac{m-\alpha}{m+2}+\frac{\nu}{2}}
(\delta_0\delta T)^{\frac{1}{2}}\| s^{-\frac{\alpha}{2}}G\|_{L^2(D_{s,y}^{\delta_0})}\leq C\delta^{\frac{m-\alpha}{m+2}} \delta_0^{\frac{1}{2}}\| s^{-\frac{\alpha}{2}}G\|_{L^2(\rb)}.
\end{split}
\label{equ:5.18}
\end{equation}

Consequently, our proof will be completed once the following inequality holds
\begin{equation}
\begin{split}
\delta^{-\frac{1}{2}+\frac{m-\alpha}{m+2}+\frac{\nu}{2}}\| t^{\frac{\alpha}{2}}v_{02}\|_{L^2(D_{t,x}^{T,\delta})}                                                                                                                                                                                                                                                                                                                                                                                                                                                                                                                                                                                                                                                                                                                                                                                                      \leq C\delta_0^{\frac{1}{2}}\| s^{-\frac{\alpha}{2}}G\|_{L^2(\ra)},
\end{split}
\label{Y-1}
\end{equation}
where
\[v_{02}=\int_1^2\int_{\mathbb{R}^n}\int_{|\xi|\geq1}e^{i\{(x-y)\cdot\xi-[\phi_m(t)-\phi_m(s)]|\xi|\}}
\frac{\beta(\delta\xi)}{|\xi|}G(s,y)\md y\md\xi \md s.\]
But \eqref{Y-1} just only follows from the estimate of \(v_{02}\) in Section \ref{sec4:w0:small} if one notes
the condtion \(\delta\leq\phi_m(T)\leq10\phi_m(2)\).

\textbf{Collecting the estimates on $w^0$ and $w^1$ in Section 4.2 and Section 4.3 respectively,
the proof of \eqref{equ:3.4.1} is completed.}

\section{Proof of Theorem 1.2}

\begin{proof}[Proof of Theorem~\ref{thm:global existence}]

Based on the smallness of the initial data
in \eqref{YH-4} , we now use the standard Picard iteration and contraction mapping principle to obtain Theorem 1.2. First let $u_{-1}\equiv0$, then for $k=0,1,2,3,\ldots$, let $u_k$ be the
weak solution of the following equation
\begin{equation}\label{equ:6.2}
\begin{cases}
&\partial_t^2 u_k-t^m\triangle u_k=t^\alpha|u_{k-1}|^p, \quad (t,x)\in(T_0, \infty)\times\mathbb{R}^n,\\
&u_k(T_0,x)=\ve f(x)\quad \partial_tu_k(T_0,x)=\ve g(x).
\end{cases}
\end{equation}

For any $p\in (p_{crit}(n,m,\alpha ), p_{conf}(n,m,\alpha )]$, one can always fix a number $\gamma>0$ satisfying
\[\frac{1}{p(p+1)}<\gamma<\frac{n}{2}-\frac{1}{m+2}-\frac{1}{p+1}\lc n+\frac{2\alpha-m}{m+2}\rc.\]
Set
\begin{align*}
M_k=&\Big\|\Big(\big(\phi_m(t)+M\big)^2-|x|^2\Big)^\gamma t^{\frac{\alpha}{q}}u_k\Big\|_{L^q(\ra)}, \\
N_k=&\Big\|\Big(\big(\phi_m(t)+M\big)^2-|x|^2\Big)^\gamma t^{\frac{\alpha}{q}}(u_k-u_{k-1})\Big\|_{L^q(\ra)},
\end{align*}
where $q=p+1$. For \(k=0\), \(u_0\) solves the linear problem
\begin{equation}\label{equ:6.2.1}
\begin{cases}
&\partial_t^2 u_0-t^m\triangle u_0=0, \quad (t,x)\in\ra,\\
&u_0(T_0,x)=f(x)\quad \partial_tu_0(T_0,x)=g(x).
\end{cases}
\end{equation}
Thus one can apply Lemma 2.1 to \eqref{equ:6.2.1}, we know that there exists a constant $C_0>0$ such that
\[M_0\leq C_0\ve.\]

Notice that for $j$, $k\geq0$,
\begin{equation*}
\begin{cases}
&\partial_t^2 (u_{k+1}-u_{j+1})-t^m \Delta (u_{k+1}-u_{j+1}) =V(u_k,u_j)(u_k-u_j),   \\
&(u_{k+1}-u_{j+1})(T_0,x)=0, \quad \partial_t(u_{k+1}-u_{j+1})(T_0,x)=0,
\end{cases}
\end{equation*}
where
\begin{equation*}
\big|V(u_k,u_j)\big|\le t^\alpha C(|u_k|+|u_j|)^{p-1}
\end{equation*}
By our assumptions
\[\gamma<\frac{n}{2}-\frac{1}{m+2}-\frac{1}{q}\lc n+\frac{2\alpha-m}{m+2}\rc \quad \text{and} \quad p\gamma>\frac{1}{q}, \quad q=p+1,\]
then Theorem~\ref{thm:useful estimate} together with H\"{o}lder's inequality yield
\begin{equation}\label{equ:6.4}
\begin{split}
&\Big\|\Big(\big(\phi_m(t)+M\big)^2-|x|^2\Big)^\gamma t^{\frac{\alpha}{q}}(u_{k+1}-u_{j+1})\Big\|_{L^q(\ra)} \\
&\leq C\Big\|\Big(\big(\phi_m(t)+M\big)^2-|x|^2\Big)^{p\gamma}t^{-\frac{\alpha}{q}}V(u_k,u_j)(u_k-u_j)\Big\|_{L^{\frac{q}{q-1}}(\ra)} \\
&\leq C\Big\|\Big(\big(\phi_m(t)+M\big)^2-|x|^2\Big)^\gamma t^{\frac{\alpha}{q}}(|u_k|+|u_j|)\Big\|_{L^q([T_0,\infty]\times\mathbb{R}^n)}^{p-1} \\
&\quad \times \Big\|\Big(\big(\phi_m(t)+M\big)^2-|x|^2\Big)^\gamma t^{\frac{\alpha}{q}}(u_k-u_j)\Big\|_{L^q(\ra)} \\
&\leq C\big(M_k+M_j\big)^{p-1}\Big\|\Big(\big(\phi_m(t)+M\big)^2-|x|^2\Big)^\gamma t^{\frac{\alpha}{q}}(u_k-u_j)\Big\|_{L^q(\ra)}.
\end{split}
\end{equation}

If $j=-1$, then $M_j=0$, and \eqref{equ:6.4} gives
\[M_{k+1}\leq M_0+\frac{M_k}{2}\quad \text{for} \quad CM_k^{p-1}\leq\frac{1}{2}.\]
This yields that
\[M_k\leq 2M_0 \quad \text{if} \quad C\big(C_0\ve\big)^{p-1}\leq\frac{1}{2}.\]
Thus we get the boundedness of $\{\big((\phi_m(t)+M)^2-|x|^2\big)^\gamma t^{\frac{\alpha}{q}}u_k\}$ in the space $L^q(\ra)$ for sufficiently small $\ve>0$. Similarly, we have
\[N_{k+1}\leq\frac{1}{2}N_k,\]
which derives that there exists a function \(u\) with $\big((\phi_m(t)+M)^2-|x|^2\big)^\gamma t^{\frac{\alpha}{q}}u\in L^q\big(\ra\big)$ such that  $\big((\phi_m(t)+M)^2-|x|^2\big)^\gamma t^{\frac{\alpha}{q}}u_k\rightarrow \big((\phi_m(t)+M)^2-|x|^2\big)^\gamma t^{\frac{\alpha}{q}}u$ in $L^q\big(\ra\big)$.
In addition, by the uniform boundedness of $\{M_k\}$, one easily calculates for any compact set \(K\subseteq\ra\),
\begin{align*}
&\| t^\alpha |u_{k+1}|^p-t^\alpha |u_{k}|^p\|_{L^{\frac{q}{q-1}}(K)} \\
&\leq C(K)\lp \Big(\big(\phi_m(t)+M\big)^2-|x|^2\Big)^{p\gamma}t^{-\frac{\alpha}{q}}(t^\alpha |u_{k+1}|^p-t^\alpha |u_{k}|^p)\rp_{L^{\frac{q}{q-1}}(K)} \\
&\leq C(K)\lp \big((\phi_m(t)+M)^2-|x|^2\big)^\gamma t^{\frac{\alpha}{q}}(u_{k+1}-u_k)\rp_{L^q(K)} \\
&\leq C(K)N_k\le C2^{-k}.
\end{align*}

Therefore $t^\alpha |u_{k}|^p\rightarrow t^\alpha |u|^p$ in $L^{\frac{q}{q-1}}\big(K\big)$ and hence in \(L^1_{loc}(\ra)\). Thus $u$ is a weak solution of (1.2)
in the sense of distributions. Then we complete the proof of Theorem 1.2.
\end{proof}

\vskip 0.2 true cm

\phantomsection
\addcontentsline{toc}{section}{Acknowledgment}

{\bf Acknowledgment.} The authors would like to thank the referee very much for
his (or her) many helpful suggestions and comments that lead to a substantial improvement
of this manuscript. The authors also would like to thank Prof. Huicheng Yin, Prof. Yi Zhou and Prof. Zhen Lei for many helpful guidance and discussions.

\appendix
\section{Appendix}\label{app}
\begin{lemma}\label{lem:a3}
\eqref{equ:4.66-1} holds.
\end{lemma}

\begin{proof}
We shall apply Lemma 3.2 in \cite{Gls1} and the dual argument to derive \eqref{equ:4.66-1}.
For $g\in L^2(\mathbb{R}^n)$,
\[\| \tilde{R}_zg\|_{L^2}=\sup_{h\in L^2}\frac{\big|(h,\overline{\tilde{R}_zg})\big|}{\| h\|_{L^2}}
=\sup_{h\in L^2}\frac{\big|(\tilde{R}_z^*h,\bar{g})\big|}{\| h\|_{L^2}}
\leq\sup_{h\in L^2}\frac{\| \tilde{R}_z^*h\|_{L^2}}{\| h\|_{L^2}}\| g\|_{L^2},\]
thus our task is reduced to estimate \(\tilde{R}_z^*h\). Since
\begin{equation*}
\begin{split}
(\tilde{R}_zg)&(t,x)=\Big(z-\frac{(m+2)n+2+2\alpha}{2(\alpha +2)}\Big)e^{z^2}
t^{\frac{\alpha}{q_0}}\phi_m(t)^{-\frac{m}{2(m+2)}}\no\\
&\times
\int_{|\xi|\geq1}\int_{\frac{1}{2}\phi_m(1)\leq|y|\leq\phi_m(2)}e^{i\{(x-y)\cdot\xi-[\phi_m(t)-|y|]|\xi|\}}\Big(1-\rho\big(\phi_m(t)^{1-\varrho}\delta^\varrho\xi\big)\Big)
g(y)\frac{\md\xi}{|\xi|^z}\md y,
\end{split}
\end{equation*}
the dual operator of $R_zg$ is
\begin{equation*}
\begin{split}
(\tilde{R}_z^*h)&(y)=\Big(\bar{z}-\frac{(m+2)n+2+2\alpha}{2(\alpha +2)}\Big)e^{\bar{z}^2}
t^{\frac{\alpha}{q_0}}\phi_m(t)^{-\frac{m}{2(m+2)}} \\
&\qquad \times\int\int e^{i\{(y-x)\cdot\xi+[\phi_m(t)-|y|]|\xi|\}}\Big(1-\rho\big(\phi(t)^{1-\alpha}\delta^\alpha\xi\big)\Big)h(x)\frac{\md\xi}{|\xi|^{\bar{z}}}\md x \\
&=\Big(\bar{z}-\frac{(m+2)n+2+2\alpha}{2(\alpha +2)}\Big)e^{\bar{z}^2}t^{\frac{\alpha}{q_0}}\phi_m(t)^{-\frac{m}{2(m+2)}} \\
&\qquad \times\int e^{i(y\cdot\xi-|y||\xi|)}e^{i\phi(t)|\xi|}\Big(1-\rho\big(\phi(t)^{1-\alpha}\delta^\alpha\xi\big)\Big)|\xi|^{-\bar{z}}\hat{h}(\xi)\md\xi.
\end{split}
\end{equation*}

Denote
\[\hat{H}(\xi)=e^{i\phi(t)|\xi|}t^{\frac{\alpha}{q_0}}\phi_m(t)^{-\frac{m}{2(m+2)}}\Big(1-\rho\big(\phi(t)^{1-\alpha}\delta^\alpha\xi\big)\Big)|\xi|^{-\bar{z}}\hat{h}(\xi).\]

Then a simple computation yields
\begin{equation*}
\begin{split}
\|\tilde{ R}_z^*h\|_{L^2(\frac{1}{2}\phi(1)\leq|y|\leq\phi(2))}& \leq C\sum_{k=0}^{\infty}2^{\frac{k}{2}}\|\hat{H}\|_{L^2(2^k\leq|\xi|\leq2^{k+1})}\qquad
(\text{by Lemma 3.2 of \cite{Gls1}}) \\
& \lesssim \sum_{2^{k+1}\leq\phi(t)^{\alpha-1}\delta^{-\alpha}}^{\infty}2^{\frac{k}{2}}\|\hat{H}\|_{L^2(2^k\leq|\xi|\leq2^{k+1})}.
\end{split}
\end{equation*}
If $k\geq0$, then for \(Rez=0\),
\begin{equation*}
\|\hat{H}\|_{L^2(2^k\leq|\xi|\leq2^{k+1})}\leq Ct^{\frac{\alpha}{q_0}}\phi_m(t)^{-\frac{m}{2(m+2)}}2^{\frac{k}{2}}\|\hat{h}\|_{L^2(2^k\leq|\xi|\leq2^{k+1})}.
\end{equation*}
Then \eqref{equ:4.66-1} follows immediately.

\end{proof}

\begin{lemma}\label{lem:a4}
\eqref{equ:4.67} holds true.
\end{lemma}

\begin{proof}
Denote $K_z$ by the kernel of the operator $S_z$. Then
\begin{equation}
\begin{split}
K_z(t;x,y)=&\Big(z-\frac{(m+2)n+2+2\alpha}{2(\alpha +2)}\Big)e^{z^2}t^{\frac{\alpha}{q_0}}\phi_m(t)^{-\frac{m}{2(m+2)}}
 \\ &\times\int_{\mathbb{R}^n}e^{i\{(x-y)\cdot\xi-[\phi_m(t)-|y|]|\xi|\}}\rho(\phi_m(t)^{1-\alpha}\delta^\alpha\xi)\frac{d\xi}{|\xi|^z} \quad
\text{with $Rez=0$}.
\end{split}\label{equ:a.5}
\end{equation}
Note that $|\xi|\geq\phi_m(t)^{\alpha-1}\delta^{-\alpha}$ holds in the integral domain of
\eqref{equ:a.5}. Therefore it follows from Lemma 3.3 in \cite{Gls1} and the condition $\delta\lesssim 10\phi_m(2)$ that for any $N\in\Bbb R^+$,
\begin{align*}
|K_z|\le C_N &t^{\frac{\alpha}{q_0}}\phi_m(t)^{-\frac{m}{2(m+2)}}\Big(\frac{\delta}{\phi_m(t)}\Big)^N\leq C_N\phi_m(t)^{\frac{2\alpha}{(m+2)q_0} -\frac{m}{2(m+2)}}(\phi_m(t)^{\alpha-1}\delta^{-\alpha})^{\frac{1}{2}} \\
&\text{if}\quad \Big||x-y|-\big|\phi_m(t)-|y|\big|\Big|\geq\frac{\delta}{2}.
\end{align*}
This yields \eqref{equ:4.67} when $\Big||x-y|-\big|\phi_m(t)-|y|\big|\Big|\geq\frac{\delta}{2}$.
When $\Big||x-y|-\big|\phi_m(t)-|y|\big|\Big|<\frac{\delta}{2}$,
analogously treated as in Lemma 3.4-Lemma 3.5 and Proposition 3.6 of \cite{Gls1},
\eqref{equ:4.67} can be also obtained, thus we omit the detail since the proof procedure
is completely similar to that in \cite{Gls1}.
\end{proof}

\begin{lemma}\label{lem:a1}
One has that for $\delta>0$,
\begin{equation}
\begin{split}
&\Big\|\int_{\Bbb R^n} e^{i(x\cdot\xi-t|\xi|)}\hat{f}(\xi)\md\xi\Big\|_{L^2(\{x:\delta\leq t-|x|\leq2\delta\})}\leq C\delta^{\frac{1}{2}}\Big(\|\hat{f}\|_{L^2(|\xi|\leq1)}
+\sum_{k=0}^{\infty}2^{\frac{k}{2}}\|\hat{f}\|_{L^2(2^k\leq|\xi|\leq2^{k+1})}\Big).
\end{split}\label{equ:a.6}
\end{equation}
\end{lemma}

\begin{proof}
The procedure of the proof is the same as that of Lemma A.5 in \cite{HWY4}, one can refer to \cite{HWY4} for details.
\end{proof}

\phantomsection

\end{document}